%% file: BNS_Surface_Houghton.tex
\documentclass[11pt]{article}

\usepackage[margin=3cm]{geometry}
\usepackage{fancyhdr}
\usepackage{amsmath,amsthm,amssymb}
\usepackage{mathrsfs}
\usepackage{textcomp}
\usepackage{graphicx}
\usepackage{accents}
\usepackage{hyperref}
\usepackage{fullpage}

\usepackage{enumerate}
\usepackage{multicol}
\usepackage{parskip}
\usepackage{fancyhdr}
\usepackage{url}
\usepackage{xspace}

\usepackage{setspace}
\usepackage{wrapfig}

\usepackage{tikz-cd}
\usepackage{tikz}
\usetikzlibrary{calc}

\usepackage{tabularx}

\theoremstyle{plain}
\newtheorem{theorem}{Theorem}[section]
\newtheorem{cor}[theorem]{Corollary}
\newtheorem{prop}[theorem]{Proposition}
\newtheorem{lemma}[theorem]{Lemma}

\newtheorem{question}[theorem]{Question}

\theoremstyle{definition}
\newtheorem{definition}[theorem]{Definition}
\newtheorem{remark}[theorem]{Remark}

\newcommand{\Z}{\ensuremath{\mathbb{Z}}\xspace}
\newcommand{\N}{\ensuremath{\mathbb{N}}\xspace}
\newcommand{\R}{\ensuremath{\mathbb{R}}\xspace}

\renewcommand{\SS}{\ensuremath{\mathcal{S}}\xspace}
\renewcommand{\O}{\ensuremath{\mathcal{O}}\xspace}
\renewcommand{\H}{\ensuremath{\mathcal{H}}\xspace}

\DeclareMathOperator{\lk}{lk}
\DeclareMathOperator{\id}{id}
\DeclareMathOperator{\Br}{Br}
\DeclareMathOperator{\Hom}{Hom}
\DeclareMathOperator{\sep}{sep}

\DeclareMathOperator{\Map}{Map}
\DeclareMathOperator{\PMap}{PMap}
\DeclareMathOperator{\Stab}{Stab}

\author{Noah Torgerson and Jeremy West}
\title{BNSR-Invariants of Surface Houghton Groups}
\date{June 7, 2024}

\begin{document}
	\maketitle

    \begin{abstract}
        The surface Houghton groups $\H_{n}$ are a family of groups generalizing Houghton groups $H_n$, which are constructed as asymptotically rigid mapping class groups. We give a complete computation of the BNSR-invariants $\Sigma^{m}(P\H_{n})$ of their intersection with the pure mapping class group. To do so, we prove that the associated Stein--Farley cube complex is CAT(0), and we adapt Zaremsky's method for computing the BNSR-invariants of the Houghton groups. As a consequence, we give a criterion for when subgroups of $H_n$ and $P\H_{n}$ having the same finiteness length as their parent group are finite index. We also discuss the failure of some of these groups to be co-Hopfian.
    \end{abstract}

	\section{Introduction}

    \indent To any group $G$ of type $F_k$, one can assign a sequence of invariants $\Sigma^{m}(G)$, for $m \leq k$. These invariants determine which subgroups of $G$ (containing the commutator subgroup) share which finiteness properties of $G$. Historically, they are difficult to compute. They were defined across several papers (see \cite{BNS}, \cite{BR}, and \cite{Renz}), primarily by Bieri, Neumann, Strebel, and Renz, hence the names ``BNS-invariants" and ``BNSR-invariants". For the remainder of this paper, these will be referred to as $\Sigma$-invariants for the sake of brevity. One recent collection of groups for which these have been computed is the family of Houghton groups, in \cite{Zar15} and \cite{Zar19}.\\
    \indent Houghton defined his groups as permutation groups of infinite sets in \cite{Houghton}. More specifically, $H_n$ is the group of ``eventual translations" of the set $\{1,\dots,n\}\times \N$, i.e. permutations which in each ray are translations outside some finite set. Brown proved, via a suitable simplicial complex, that Houghton's group $H_{n}$ is of type $F_{n-1}$, but not of type $F_{n}$ (in the language of ``finiteness length" this is $fl(H_{n})=n-1$)\footnote{Brown proves type $FP_{n-1}$ but not type $FP_{n}$, and finitely presented, which together imply $F_n$ but not $F_{n-1}$.}. In \cite{ABKL}, the authors define a variant of Houghton groups, called surface Houghton groups, as asymptotically rigid mapping class groups of surfaces with infinite genus. We denote the surface Houghton groups by $\H_{n}$, and the pure (i.e. end-fixing) subgroups as $P\H_{n}$. In the same paper, they also show that $fl(\H_{n})=n-1$, via a cube complex analogous to the Stein--Farley complexes for Thompson groups.\\
    \indent It is worth noting that this is not the only asymptotically rigid mapping class group variant of Houghton groups. While the surface Houghton groups replace the $\N$-rays with ends accumulated by genus (without boundary), the braided Houghton groups (defined by Degenhardt in his PhD thesis \cite{degenhardt}) can be realized as asymptotically rigid mapping class groups as well (see \cite{Funar} for details). This construction replaces the $\N$-rays with planar ends, accumulated by punctures; note that this surface has non-compact boundary. These braided Houghton groups $\Br H_{n}$ also have $fl(\Br H_{n})=n-1$ (as proven in \cite{GLU1}). The braided Houghton groups shall play a small role in the final section of this paper.\\
    \indent In \cite{Zar15} and \cite{Zar19}, Zaremsky computed the $\Sigma$-invariants of the Houghton groups, using the cube complex defined implicitly in \cite{Brown}, and explicitly in \cite{Lee}. We carry out a parallel of Zaremsky's arguments, showing that the equivalent statement holds for pure surface Houghton groups. Namely, we prove in Section 4 the following, where $m(\chi)$ is the number of non-zero coefficients of $\chi$ in ascending standard form (details in Section 4):

    \setcounter{section}{4}
    \setcounter{theorem}{0} % 1 - 1
    \begin{theorem}
         Let $\chi$ be a non-zero character of $P\H_{n}$. Then $[\chi]\in \Sigma^{m(\chi)-1}(P\H_{n})\setminus \Sigma^{m(\chi)}(P\H_{n})$.
    \end{theorem}
    \setcounter{section}{1}
    \setcounter{theorem}{0} % 1 - 1

    \indent As an application of this, we provide a partial converse (for Houghton groups and surface Houghton groups) of the well-known fact that if $H\leq G$ is a finite index subgroup, then $fl(H)=fl(G)$. Specifically, we show that whenever a subgroup of $H_{n}$ or $P\H_{n}$ has finiteness length $n-1$ and intersects the corresponding commutator with finite index, it is finite index in the full group. Further, in order to be finite index, the intersection condition must hold, for general group theoretic reasons. (These are also true of $\H_{n}$, but for trivial reasons: the commutator subgroup of $\H_{n}$ is all of $\H_{n}$.) This is carried out in Section 5, along with some discussion of co-Hopfianness.\\
    \indent In order to accomplish this, we demonstrate in Section 3 that the cube complex of \cite{ABKL} is CAT(0).

    \setcounter{section}{3}
    \setcounter{theorem}{8} % 9 - 1
    \begin{theorem}
        The Stein--Farley complex for $\H_{n}$ is CAT(0).
    \end{theorem}
    \setcounter{section}{1}
    \setcounter{theorem}{0} % 2 - 1
    
    To do this, we use a refinement of a proposition from \cite{GLU1}, see Proposition \ref{CATProp}. In Section 2, we lay out the definitions of the surfaces and groups we are concerned with, of the corresponding cube complex, and of the version of discrete Morse theory we shall employ. Alongside these definitions are various lemmas which we shall need. We also obtain various nice representatives for the vertices (Section 2) and edges (Section 3) of the complex.\\
    \indent In Section 5, we finish with some applications of the $\Sigma$-invariants to when subgroups having maximal finiteness length have finite index. This discussion applies to $H_n$, $\H_n$, and $P\H_n$. To handle the infinite index case, there is some discussion of the failure of these groups to be co-Hopfian. In particular, we have the following theorem.

    \setcounter{section}{5}
    \setcounter{theorem}{5} % 6 - 1
    \begin{theorem}
        Let $H$ denote either the Houghton group, or the pure surface Houghton group, and suppose $G<H$ has $fl(G)=fl(H)$. Then $G$ is finite index in $H$ if and only if $G\cap H'$ is finite index in $H'$, where $H'$ denotes the commutator subgroup of $H$. Furthermore, there exist subgroups $G$ with $fl(G)=fl(H)$ of both finite and infinite index.
    \end{theorem}
    \setcounter{section}{1}
    \setcounter{theorem}{0} % 1 - 1
    
    \indent A recent paper by Marie Abadie \cite{Aba} analyzes the CAT(0) property for the Stein--Farley cube complexes of a different family of asymptotically rigid mapping class groups, including the braided Houghton groups. Similar methods are used, in particular a version of the same proposition as we use. Thus, one could attempt to apply Zaremsky's methods to compute the $\Sigma$-invariants of $\Br H_{n}$ as well. It seems reasonable to guess that they should work out the same.\\
    \indent The recent paper \cite{ADL} concerns a generalization of these surface Houghton groups, obtained by varying the rigid structure on the same surface. As they show that the groups they consider are finite index subgroups of our $\H_n$, the $\Sigma$-invariants are effectively the same. Additionally, our results in Section 5 also apply to these more general surface Houghton groups.
    \section*{Acknowledgements}\indent We would like to thank Jing Tao for advice and guidance on this project. We would also like to thank Noel Brady and Justin Malestein for valuable conversations, and Peter Patzt and Matt Zaremsky for helpful comments. We also thank the anonymous referee for helpful comments. The authors were partially supported by the OU Bridge Funding Investment Program.
    \section{Definitions}

    \subsection{The (Pure) Surface Houghton Group} Here we lay out the definitions necessary for the group $P\H_{n}$, the pure version of the surface Houghton group defined in \cite{ABKL}. Let $\O=\O_{n}$ be a sphere with $n$ boundary components, and let $T$ be a torus with two boundary components, $\partial^-$ and $\partial^+$. Fix an orientation-reversing homeomorphism $\lambda:\partial^-\rightarrow\partial^+$, and for each $i$ an orientation-reversing homeomorphism $\mu_{i}$ from $\partial^-$ to the $i$th boundary component of $\O$. We construct $\Sigma_{n}$ as follows: begin with $M^{1}=\O$, then glue a copy of $T$ to each boundary component of $M^{1}$ via the $\mu_i$ to obtain $M^{2}$. For each $j\geq 2$, glue a copy of $T$ to each boundary component of $M^{j}$ via $\lambda$ to obtain $M^{j+1}$. Then the surface $\Sigma_{n}$ is the union of all the $M^{j}$'s; $\Sigma_n$ is the surface with $n$ ends, all accumulated by genus. We call $\O$ the \textbf{center} of $\Sigma_n$, each of the closures of the components of $M^{j}\setminus M^{j-1}$ is called a \textbf{piece}, and each piece $B$ has a canonical homeomorphism $\iota_{B}:B\rightarrow T$. We will occasionally write $B_{k}^{j}$ to denote the $j$th piece in the $k$th end.\\
    \indent Call a subsurface of $\Sigma_{n}$ \textbf{suited} if it is connected and the union of $\O$ and finitely many pieces. Let $\varphi:\Sigma_{n}\rightarrow\Sigma_{n}$ be a homeomorphism. Call $\varphi$ \textbf{asymptotically rigid} if there exists a suited subsurface $Z\subset \Sigma_{n}$ (called a \textbf{defining surface} for $\varphi$) such that
    \begin{itemize}
        \setlength\itemsep{-.4em}
        \item[$\bullet$] $\varphi(Z)$ is also suited, and 
        \item[$\bullet$] $\varphi$ is \textbf{rigid away from $Z$}, that is, for every piece $B\subset \overline{\Sigma_{n}\setminus Z}$, we have that $\varphi(B)$ is a piece, and $\varphi|_{B} = \iota_{\varphi(B)}^{-1}\circ \iota_{B}$.
    \end{itemize}
    We may sometimes say that $\varphi$ is rigid away from a piece $B$ adjacent to $\O$, which we take to mean that $\varphi$ is rigid away from $\O\cup B$. A special family of suited subsurfaces are those with pieces in only one end. We denote by $L_{g}$ the suited subsurface consisting of $\O$ and the first $g$-many pieces in the $n$th end.
    \begin{definition}
        \textbf{(Surface Houghton Group)} The Surface Houghton Group $\H_{n}$ is the subgroup of the mapping class group $\Map(\Sigma_{n})$ whose elements have an asymptotically rigid representative. The Pure Surface Houghton Group $P\H_{n}$ is the intersection of $\H_{n}$ with the pure mapping class group $\PMap(\Sigma_{n})$.
    \end{definition}

    \indent We define some special elements of $\H_{n}$. For $\sigma$ a permutation of $\{1,\dots,n\}$, choose some homeomorphism of $\O$ permuting the boundary components in the same fashion. For the sake of definiteness, consider $\sigma$ first as an element of $\Map(S_{0,n})$, the sphere with $n$ punctures; then $\sigma$ yields a homeomorphism of $\O$, which we also denote by $\sigma$. Up to some isotopy, we can assume that the restriction of $\sigma$ to the $i$th boundary component of $\O$ is $\mu_{\sigma(i)}\circ\mu_{i}^{-1}$. Thus, we can obtain a homeomoprhism of $\Sigma_n$ by extending this map to be rigid outside of $\O$; call this extension again by the same name, $\sigma$. Note that such a map is not unique: there are many choices of mapping class in the sphere with $n$ punctures which induce the same permutation.\\
    \indent Next, we define the handle shifts $\rho_{i}$, $\grave{a}$ $la$ \cite{PV}. For $i\in \{1,\dots,n-1\}$, we want $\rho_{i}$ to be a homeomorphism of $\Sigma_{n}$ which shifts the $i$th end towards the $n$th end by a single genus. Specifically, for $j\geq 1$ map each piece $B_{n}^{j}$ to $B_{n}^{j+1}$, and map each piece $B_{i}^{j+1}$ to $B_{i}^{j}$; choose these to be the rigid maps. The ends other than $i$ and $n$ are unchanged, so all that remains is to define how $\rho_{i}$ acts on $B_{i}^{1}\cup \O$. This is demonstrated by Figure \ref{handle shift}. Note that $\rho_{i}$ is asymptotically rigid outside of $\O\cup B_{i}^{1}$. \\

    \begin{figure}[h!]
        \include{handle-shift}
        \caption{The handle shift $\rho_i$ (which pushes from end $i$ to end $n$) being applied to various curves in ends $i$, $n$, and $j$.}
        \label{handle shift}
    \end{figure}
    
    \begin{lemma}
    \label{map form}
        Any map $\varphi\in \H_{n}$ can be written as $\alpha\rho_{i_{1}}\cdots\rho_{i_{j}}\sigma$, where $\alpha$ is a compactly supported mapping class, and $\sigma$ and the $\rho_{i_{k}}$'s are as above.
    \end{lemma}
    \begin{proof}
    Consider an arbitrary element $\varphi\in\H_{n}$: $\varphi$ clearly induces some permutation $\sigma$ on the ends of $\Sigma_{n}$; obtain from this $\sigma\in\H_{n}$. Now we have that $\varphi\circ\sigma^{-1}\in P\H_{n}$, so we wish to analyze how pure mapping classes act within ends. A rigorous examination of this is contained in Section 2.3. In the meantime, consider how $\varphi$ acts in the complement of a defining surface: as it does not permute the ends, and takes pieces to pieces, it must shift each end in or out by some integer amount. From this, we obtain a sequence of handle shifts $\rho_{i_{1}},\dots,\rho_{i_{k}}$. Finally, we have that $\varphi\circ \sigma^{-1}\circ (\rho_{i_{1}}\cdots\rho_{i_{k}})^{-1}$ is a compactly supported mapping class, $\alpha$. 
    \qedhere
    \end{proof}
    \begin{remark}
        The above expression obtained for $\varphi$ is far from unique: choices were made both in the selection of $\sigma$, as well as in the selection of the handle shifts. However, any different choices made for these will simply change $\alpha$, and the properties of $\alpha$ are rarely relevant for this paper.
    \end{remark}

    \subsection{The Contractible Cube Complex} The group $\H_{n}$ acts on a contractible cube complex $X_{n}$, called the Stein--Farley complex. Consider ordered pairs $(Z,\varphi)$, where $Z$ is a suited subsurface, and $\varphi\in \H_{n}$. Declare two such pairs $(Z_{1},\varphi_{1})$ and $(Z_{2},\varphi_{2})$ to be equivalent if the \textbf{transition map} $\varphi_{2}^{-1}\circ \varphi_{1}$ takes $Z_{1}$ to $Z_{2}$ homeomorphically, and is rigid elsewhere (i.e. in the complement of $Z_{1}$).\\
    \indent Denote the equivalence class of $(Z,\varphi)$ by $[Z,\varphi]$, and the set of equivalence classes by $\SS$. The group $\H_{n}$ acts on $\SS$ by $$\psi\cdot[Z,\varphi] = [Z,\psi\circ \varphi].$$
    Define the \textbf{complexity} of a pair $(Z,\varphi)$ to be the genus of $Z$. This extends to be an $\H_{n}$-invariant height function $f\colon\SS\rightarrow \Z\subset\R$. Given vertices $x_{1},x_{2}\in \SS$, we say that $x_{1}\prec x_{2}$ if there are representatives $(Z_{i},\varphi_{i})$ of $x_{i}$ so that $\varphi_{1} = \varphi_{2}$, $Z_{1}\subset Z_{2}$, and $\overline{Z_{2}\setminus Z_{1}}$ is a (non-empty) \textit{disjoint} union of pieces (i.e. has at most one piece from each end).\\
    \indent We now construct $X_{n}$ as a cube complex with vertex set $\SS$. We declare that whenever $x_{1}\prec x_{2}$, with $d=f(x_{2})-f(x_{1})$, there is a $d$-cube with vertex set given by $\{x\,|\, x_1 \preceq x\preceq x_2\}$. Note that as $\Sigma_{n}$ has $n$ ends, the complex $X_{n}$ is $n$ dimensional. The height function $f$ extends affinely to an $\H_{n}$-invariant height function on $X_{n}$. Also, let $X_n^{f\leq k}$ denote the subcomplex spanned by vertices with height at most $k$. We have the following properties (see \cite[Theorem 4.1]{ABKL}):
    \begin{theorem}
        The cube complex $X_n$ is contractible, and the action of $\H_{n}$ satisfies:
        \begin{itemize}
            \setlength\itemsep{-.4em}
            \item[$\bullet$] The $\H_{n}$-stabilizers of cubes are finite extensions of mapping class groups of compact surfaces.
            \item[$\bullet$] For $k\geq 1$, the subcomplex $X_n^{f\leq k}$ is $\H_{n}$-cocompact.
        \end{itemize}
    \end{theorem}
    \indent We are actually more concerned with the action of $P\H_{n}$ on $X_n$. Note that $P\H_{n}$ is a finite index subgroup of $\H_{n}$, and that the finite extension in the above theorem is by permutations of the ends. As mapping class groups of compact surfaces are of type $F_{\infty}$, we have the following.
    \begin{cor}
        The action of $P\H_{n}$ on $X_n$ satisfies:
        \begin{itemize}
        \setlength\itemsep{-.4em}
            \item[$\bullet$] The $P\H_{n}$-stabilizers of cubes are type $F_{\infty}$.
            \item[$\bullet$] For $k\geq 1$, the subcomplex $X_n^{f\leq k}$ is $P\H_{n}$-cocompact.
        \end{itemize}
    \end{cor}

	\indent We end this section by demonstrating that vertices and edges can be given a nice form. First, we show that every vertex (indeed, every edge) can be expressed with only elements of $P\H_{n}$, then we provide a canonical form for the suited subsurface component of the vertex. Specifically, we show that any edge can be represented using only elements in the pure group:
    \begin{lemma}
        Any edge in $X_n$ can be expressed as $[Z,\varphi]$\----$[Z\cup B, \varphi]$, where $\varphi\in P\H_{n}$.
    \end{lemma}
    \begin{proof}
        Consider a vertex $[Z,\varphi]$, and let $B$ be a piece adjacent to $Z$, so that $Z\cup B$ is suited. Write $\varphi = \alpha\rho_{i_{1}}\cdots\rho_{i_{j}}\sigma$ as in Lemma \ref{map form}. Then we wish to show that $[Z,\varphi] = [\sigma(Z),\varphi\circ\sigma^{-1}]$, and that $[Z\cup B, \varphi] = [\sigma(Z)\cup \sigma(B), \varphi\circ\sigma^{-1}]$. Consider the following diagram.

         
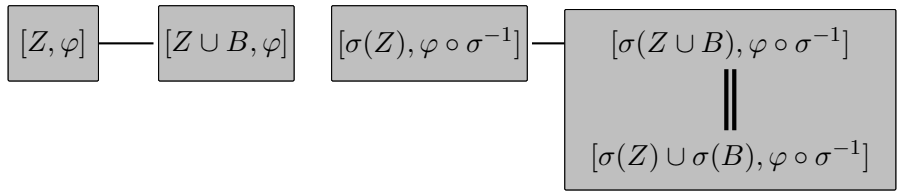
\begin{figure}[h!]
        \include{2.6-diagram}
    
       \end{figure}

    \vspace*{-.6cm}
    As $\sigma$ is rigid away from the center $\O$, we have that $\sigma(B)$ is a piece, and of course it must be adjacent to $\sigma(Z)$, so all that remains to be checked is that $[Z,\varphi] = [\sigma(Z),\varphi\circ \sigma^{-1}]$. First, note that $(\varphi\circ \sigma^{-1})^{-1}\circ \varphi = \sigma$, which is rigid away from $\O$, and thus takes $Z$ to $\sigma(Z)$ (a suited subsurface), and is certainly rigid outside of $Z$.
    \qedhere
    \end{proof}
    We assume from here on that any map $\varphi$ is in $P\H_{n}$. Vertices have two components, a map and a surface. We now have some control over the map, but what of the surface? For many arguments, it will be convenient to have the following canonical form. Recall that $L_{g}$ is the surface built by taking the center, and adding in the first $g$-many pieces in the $n$th end. Then any edge can be represented with both surfaces an appropriate $L_{g}$:
    \begin{lemma}
    \label{nicerep}
        Any edge in $X_n$ can be expressed as $[L_{g},\varphi]$\----$[L_{g+1},\psi]$.
    \end{lemma} 
    \begin{proof}
    Suppose we are given a vertex $[Z,\varphi']$. Recall that the handle shift $\rho_{i}$ is rigid away from $B_{i}^{1}$, the piece that it pushes across the center. As $(\varphi)^{-1}\circ \varphi'$ can be chosen to be a composition of handle shifts taking $Z$ to $L_{g}$, the only places where $(\varphi)^{-1}\circ \varphi'$ can be non-rigid are pieces in $Z$. This shows that vertices can be represented as $[L_{g},\varphi]$. To extend this to the edge $[Z,\varphi']$ \----$[Z\cup B,\varphi']$, consider the following diagram.

    
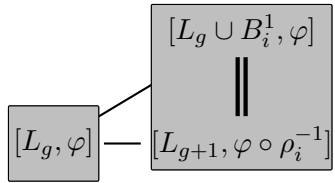
\begin{figure}[h!]
        \include{2.7-diagram}
    \end{figure}

    \vspace*{-.6cm}
    Observe that to obtain an $L_{g+1}$ form for $[Z\cup B,\varphi']$, we can first push $Z$ to $L_{g}$, which sends $B$ to a piece $B_{i}^{1}$, and then append one more $\rho_{i}$. 
    \qedhere
    \end{proof}

    \subsection{Characters and $\Sigma$-Invariants}
	
    \indent Recall that a group is of \textbf{type F$_{m}$} if it admits a proper cocompact action on an $(m-1)$-connected CW-complex. Assume $Y$ is a CW complex and $h:Y\rightarrow \R$ is continuous. We call the corresponding filtration $\left(Y^{t\leq h}\right)_{t\in \R}$ on $Y$ \textbf{essentially ($m-1$)-connected} if for any $t\in \R$, there exists $s\leq t$ such that inclusion $Y^{t\leq h}\hookrightarrow Y^{s\leq h}$ induces the trivial map in $\pi_{k}$ for $k\leq m-1$.\\
    \indent For any group $G$, we define a \textbf{character} of $G$ to be a homomorphism from $G$ to the additive group $\R$. As characters factor through $G^{ab}$ (the abelianization of $G$), the space of characters is a vector space with the same dimension as the rank of $G^{ab}$. Excluding the trivial character and modding out by positive scaling, we obtain the \textbf{Character Sphere} $\Sigma(G)$. The $\Sigma$-invariants (also called BNS or BNSR invariants, for Bieri-Neumann-Strebel(-Renz)) are a filtration of the character sphere into subspaces $$\Sigma^{0}(G)\supseteq \Sigma^{1}(G)\supseteq \Sigma^{2}(G) \supseteq\cdots,$$ defined as follows (we use a slight correction of the definition in \cite{Zar15}, as suggested in \cite{Zar19}). 
    \begin{definition}
        Let $G$ be a group of type $F_{m}$, and let $Y$ be an $(m-1)$-connected CW complex on which $G$ acts cocompactly. Suppose that the stabilizer of any $k$-cell is of type $F_{m-k}$ and is contained in the kernel of every character of $G$.\footnote{That the cell-stabilizers are contained in the kernels is the modification.} For each non-trivial $\chi\in \Hom(G,\R)$, there is a character height function $h_{\chi}:Y\rightarrow \R$, a continuous map satisfying $h_{\chi}(gy) = \chi(g)+h_{\chi}(y)$ for all $y\in Y$ and $g\in G$. Then $\Sigma^{m}(G)$ is the set of those $[\chi]$ such that the filtration $(Y^{t\leq h_{\chi}})_{t\in \R}$ is essentially $(m-1)$-connected.
    \end{definition}

    \indent Of what use are these invariants, one might ask? A standard result is the following classification of when finiteness properties of $G$ are preserved to subgroups containing the commutator subgroup.
    \begin{prop}
        Let $G$ be a group of type $F_m$, and let $K$ be a subgroup so that $[G,G]\leq K\leq G$. Then $K$ is of type $F_{m}$ if and only if for every character $\chi\in\Hom(G,\R)$ such that $\chi(K)=0$, we have $[\chi]\in \Sigma^{m}(G)$.
    \end{prop}

    \begin{definition}
        A \textbf{(Zaremsky-)Morse function} on an affine cell complex $Y$ is a map $h=(\chi,f):Y\rightarrow \R\times \R$ such that both $\chi$ and $f$ are affine on cells. The codomain is ordered lexicographically, and we require that $f$ take only finitely many values on $Y^{(0)}$, and that there is some $\varepsilon>0$ such that adjacent vertices $v$ and $w$ satisfy either $|\chi(v)-\chi(w)|\geq \varepsilon$, or $\chi(v)=\chi(w)$ and $f(v)\neq f(w)$.
    \end{definition}

    We record the following general form of the Morse Lemma, as well as a more specific version of it. Both will be used in Section 4.
    \begin{lemma}
    \label{Morse_Lemma}
        Let $-\infty\leq p\leq q\leq r \leq +\infty$. If for every vertex $v\in Y^{q<\chi\leq r}$ the descending link $\lk_{Y^{p\leq \chi}}^{h\downarrow}(v)$ is $(k-1)$-connected, then the pair $(Y^{p\leq \chi\leq r}, Y^{p\leq \chi\leq q})$ is $k$-connected. If for every vertex $v\in Y^{p\leq \chi< q}$ the ascending link $\lk_{Y^{\chi\leq r}}^{h\uparrow}(v)$ is $(k-1)$-connected, then the pair $(Y^{p\leq \chi\leq r}, Y^{q\leq \chi\leq r})$ is $k$-connected. 
    \end{lemma}

    \begin{cor}
        Let $h=(\chi,f):Y\rightarrow \R\times\R$ be a Morse function. If $Y$ is $(m-1)$-connected and for every vertex $v\in Y^{\chi<q}$ the ascending link $\lk_{Y}^{h\uparrow}(v)$ is $(m-1)$-connected, then $Y^{q\leq \chi}$ is $(m-1)$-connected.
    \end{cor}

    \indent Informally, we can think of the characters on $P\H_{n}$ as being generated by counting the handle shifts $\rho_{1},\dots,\rho_{n-1}$, where $\rho_{i}$ shifts the $i$th end into the $n$th end by one piece. Specifically, for $1\leq i\leq n-1$, take $\chi_{i}(\varphi)$ to be the negative of the sum of the powers of $\rho_{i}$ appearing in $\varphi$; for $i=n$, take $\chi_{n}$ to be the sum of the powers of $\rho_{1}$ through $\rho_{n-1}$. We stress that this definition, while convenient to use, is not obviously well-defined. An equivalent definition, which is better suited to demonstrating well-definedness, is as follows (details can be found in Section 3 of \cite{APV}).\\
    \indent Let $\gamma$ be an oriented curve that separates one end $E$ of $\Sigma_{n}$ from the rest, oriented so that the end $E$ is on the righthand side of $\gamma$. Then $\gamma$ defines a non-zero element of $H_{1}^{\sep}(\Sigma_{n},\Z)$. To every $\varphi\in\PMap(\Sigma_{n})$ and $\gamma\in H_{1}^{\sep}(\Sigma_{n},\Z)$, associate an integer $\theta_{[\gamma]}(\varphi)$, as a ``signed genus" between $\gamma$ and $\varphi(\gamma)$. Then the map $\theta_{[\gamma]}:\PMap(\Sigma_{n})\rightarrow \Z$ is a well-defined nontrivial homomorphism, depending only on the homology class of $\gamma$. By identifying $H_{\sep}^{1}$ with $\Hom(H_{1}^{\sep},\Z)$ via the Universal Coefficients Theorem, we obtain a map $\Theta:\PMap(\Sigma_{n})\rightarrow H_{\sep}^{1}(\Sigma_{n},\Z)$, by the rule $\Theta(\varphi)[\gamma] = \theta_{[\gamma]}(\varphi)$. Restricting to $P\H_{n}$, we obtain characters, and it is not difficult to see that they agree with the informal definition above. We emphasize: $\chi_{i}$ measures how much $\varphi$ pushes the $i$th end \underline{out}.\\
    \indent Given a vertex, we can define the character height function $h_{\chi_{i}}([Z,\varphi])$ by adding the number of pieces of $Z$ in the $i$th end to $\chi_{i}(\varphi)$. This is well defined: consider two representatives $[Z,\varphi] = [W,\psi]$. We have that $\varphi^{-1}\psi$ takes $W$ to $Z$, and is rigid elsewhere. This composition can be further decomposed as a mapping class of some suited subsurface composed with a sequence of handle shifts. This compactly supported mapping class affects neither the distribution of the pieces of $Z$ nor the value of any character on $\varphi$, so we need consider only handle shifts. Consider $[Z,\varphi] = [\rho_{i}(Z),\varphi\circ\rho_{i}^{-1}]$: we have moved one piece from the $i$th end to the $n$th end, while increasing the amount of pushing into the $i$th end and out of the $n$th end by one. These cancel out in $h_{\chi_{i}}$ and $h_{\chi_{n}}$. All other basis characters are unchanged under these operations, and thus we have well-defined character height functions on $X_n$. For an arbitrary character $\chi$, we shall henceforth abuse notation by writing $\chi$ when we mean $h_{\chi}$. The domain will typically be clear.

    \section{The Stein--Farley complex $X_n$ is CAT(0)}

    We have a contractible cube complex on which $P\H_{n}$ acts nicely. In this section, we show that $X_n$ is in fact CAT(0). To begin, we require the following definition.
    \begin{definition}
        A cube complex $X$ is \textbf{cube-complete} if whenever $X^{(1)}$ contains an embedded copy of the 1-skeleton of a $d$-cube, $X$ contains the entire $d$-cube.
    \end{definition}

    Once we have shown that $X_n$ is cube-complete, we seek to apply the following proposition, which is a version of \cite[Proposition 4.6]{GLU1}, where it is taken out of the proof of \cite[Theorem 6.1]{Chepoi}. 
	
	\begin{prop}
 \label{CATProp}
        Let $X$ be a cube-complete cube complex, whose 1-skeleton $X^{(1)}$ is a graph with no loops or bigons. Suppose that 
        
        \begin{enumerate}[(a)]
		      \item $X$ is simply connected,
		      \item $X^{(1)}$ satisfies the $3$-square condition.
        \end{enumerate} Then $X$ is CAT(0).
    \end{prop}

    \indent The $3$-square condition says that whenever the cube complex has 3 squares intersecting in a vertex, and pairwise intersecting in edges (call such an arrangement a 3-wheel), then they are part of a full $3$-cube.
    \begin{remark}
         In \cite{GLU1}, there is a third condition: that the 1 skeleton not contain a copy of the complete bipartite graph $K_{2,3}$. This third condition is redundant: A careful examination of the proof of Theorem 6.1 in \cite{Chepoi}, especially the paragraph beginning ``To prove the converse", yields that the other hypotheses imply the non-existence of $K_{2,3}$'s: A $K_{2,3}$ is a 3-wheel where all the vertices opposite the common vertex are identified, which cannot be the case in an embedded 3-cube. Another, smaller modification: there was a condition that the 1-skeleton be a connected graph, which is automatically true in a simply connected cube complex.
    \end{remark}
    
	As we have a height function for $X_n$, and as adjacent vertices always differ by exactly 1, there are two ways in which $X_n$ can fail to be cube-complete.
    \begin{definition}
        By a \textbf{collapsed} cube, we mean the 1-skeleton of a $d$-cube, such that the difference between its highest and lowest heights is less than $d$. By an \textbf{empty} cube, we mean the 1-skeleton of a $d$-cube that is not filled by a $d$-cube.
    \end{definition} Given the definition of $X_n$, we need to check two things: that $X_n^{(1)}$ has no collapsed cubes, and that $X_n^{(1)}$ has no empty cubes.\\
	\indent Choosing vertex representatives to have their homeomorphism components contained in the pure group allows us a consistent definition of ``direction" for adding and removing pieces, i.e. we may consistently label the ends from $1$ to $n$ without worrying about any potential re-indexing.\\ 
    \indent In order to show that $X_n$ is cube-complete, we shall proceed as follows: first, we show that given a vertex $[Z,\varphi]$, there is only one ascending edge per direction (see Lemma \ref{UAE}). Secondly, we show that there are no collapsed squares. This easily gives that there are no collapsed cubes of any dimension. Finally, we show that there are no empty squares, and from this induct to show that there are no empty cubes.

    \newpage
    \begin{lemma}
    \label{UAE}
        \textbf{(Uniqueness of Ascending Edges)} Given a vertex $v=[Z,\varphi]$, there exists exactly one ascending edge in each direction, i.e. for $i\in\{1,\dots,n\}$, the collection $\{[Z\cup B_{i},\varphi]\}$, where $B_{i}$ is the piece in the $i$th end adjacent to $Z$, exhausts the vertices of the ascending link of $v$.
    \end{lemma} \begin{proof} What we wish to show is that, given two representatives $[Z,\varphi]=[Z',\varphi']$, whenever we ascend from $[Z',\varphi']$ by adding a piece in the $i$th end, we obtain the same vertex as by adding a piece in the $i$th end to $[Z,\varphi]$. Consider the following diagram: 

    
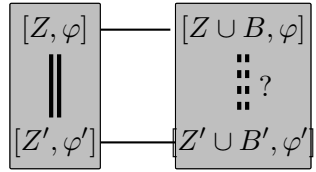
\begin{figure}[h!]
        \include{3.5-digram}
    \end{figure}

    \vspace*{-.6cm}
    Our goal is to show that, assuming $B'$ and $B$ are both in the $i$th end, we have $[Z'\cup B',\varphi'] = [Z\cup B,\varphi]$. As $(\varphi')^{-1}\circ \varphi$ is rigid outside $Z$, it takes $B$ to some piece adjacent to $Z'$. As $(\varphi')^{-1}\circ \varphi \in P\H_n$, this implies that $(\varphi')^{-1}\circ \varphi(B)$ must still be in the $i$th end (elements of $P\H_n$ cannot rigidly move pieces from one end into another), and hence is $B'$. We then have that $(\varphi')^{-1}\circ \varphi$ takes $Z\cup B$ to $Z'\cup B'$, and is rigid elsewhere, i.e. that $[Z\cup B,\varphi] = [Z'\cup B', \varphi']$. Thus, there can be only one ascending edge per direction.
    \qedhere
    \end{proof}
    \indent Combining this with the $L_{g}$ representatives of Lemma \ref{nicerep}, we see that the ascending edges over $[L_{g},\varphi]$ are of the form $[L_{g+1},\varphi\rho_{i}^{-1}]$, for $i\in\{1,\dots,n\}$. (Recall that $\rho_{n}$ is just the identity.) Next, we prove that there are no collapsed squares.

    \begin{lemma} \label{noCollSq}
        Given any square in $X_n^{(1)}$, the difference between the maximal and minimal heights of vertices is 2.
    \end{lemma} \begin{proof} For the sake of contradiction, suppose there is some collapsed square, which has only two values for the heights of its vertices. Writing the lower height vertices as $[L_{g},\varphi]$ and $[L_{g},\psi]$, we see that a collapsed square must occur as in Figure \ref{collapsed square}. That the same representative can be used for both edges from one of the side vertices is justified by Lemma \ref{UAE}.\\
    \begin{figure} 
        \include{figure-1-collapsed-square}
    \vspace*{-.6cm}
    \caption{A collapsed square}
    \label{collapsed square}
    \end{figure}
    To show that the collapsed square in Figure \ref{collapsed square} cannot occur, we assume that the middle pair of vertices are distinct, and prove that the leftmost and rightmost vertices cannot be distinct. The equality of the different forms of the middle vertices says that the transition map $\rho_{i_{k}}\varphi^{-1}\psi\rho_{j_{k}}^{-1}$ takes $L_{g+1}$ to itself and is rigid elsewhere, for $k=1,2$. This implies that $\rho_{i_{k}}\varphi^{-1}\psi\rho_{j_{k}}^{-1}$ has net-zero shifting in all ends, and hence is compactly supported (recall that for asymptotically rigid maps, these are the same). As the middle vertices are distinct, we know that $i_{1}\neq i_{2}$ and $j_{1}\neq j_{2}$. Thus, we see that $i_{k}=j_{k}$ for $k=1,2$.\\
    \indent We will now show that $[L_{g},\varphi] = [L_{g},\psi]$ by considering the transition map $\varphi^{-1}\psi$ on pieces outside of $L_{g}$. The rigidity of the $\rho_{i_{k}}\varphi^{-1}\psi\rho_{j_{k}}^{-1}$ maps outside $L_{g+1}$ immediately handles all but three pieces: $B_{i_{1}}^{1}$, $B_{i_{2}}^{1}$, and $B_{n}^{g+1}$. The last of these is again easy: the $L_{g+1}$ transition maps are rigid on the $g+2$nd piece, which they first push to the $g+1$st piece, then apply $\varphi^{-1}\psi$, then push back out by one. As this must be sent rigidly to the $g+2$nd piece, we see that $\varphi^{-1}\psi$ must be rigid on the $g+1$st piece. For $B_{i_{1}}^{1}$, consider that $\varphi^{-1}\psi$ and $\rho_{i_{2}}\varphi^{-1}\psi\rho_{i_{2}}^{-1}$ act the same on it. As the latter is rigid there, so is the former. Symmetrically, we have rigidity on $B_{i_{2}}^{1}$, and are done. Specifically, we have shown that $\varphi^{-1}\psi$ takes $L_{g}$ to itself, and is rigid elsewhere, implying that the collapsed square we began with was degenerate to begin with.\\
    \qedhere
    \end{proof}

    \begin{prop}\label{cubeComp}
        The Stein--Farley complex $X_n$ is cube complete.
    \end{prop}
 \begin{proof} We begin by observing that there can be no collapsed cubes: any such cube must contain a collapsed square. Now we show that there are no empty squares. As there are no collapsed squares, we know that the 1-skeleton of a square must have a lowest vertex, say at height $g$, two vertices at height $g+1$, and one vertex at height $g+2$. By the uniqueness (per end) of ascending edges (Lemma \ref{UAE}), we see that given a representative $[Z,\varphi]$ for the height $g$ vertex, the other vertices must be of the form $[-,\varphi]$, with the blank filled by $Z\cup B_{1}$, $Z\cup B_{2}$, and $Z\cup B_{1}\cup B_{2}$, depending on height. The only concern about the $g+2$ vertex would be that the pieces lie in the same end, but we see that this cannot occur. Thus, there are no empty squares.\\
    \indent Finally, we show that there are no empty cubes of any dimension. We do so inductively: suppose we already have that there are no empty $(d-1)$-cubes. Consider the $1$-skeleton of a $d$-cube $C$. Let the bottom vertex of $C$ be $[Z,\varphi]$, at height $g$. Each of the vertices of $C$ adjacent to $[Z,\varphi]$ is of the form $[Z\cup B_{i},\varphi]$, where $B_{i}$ is in the $i$th end; each $i$ appears at most once by Lemma \ref{UAE}. Also, by the inductive hypothesis, we have that $[Z,\varphi]$ connects to each of the vertices at height $g+d-1$ via a $(d-1)$-cube, so that these vertices can be written as $[Z\cup B_{i_{1}}\cup \cdots B_{i_{j-1}}\cup \widehat{B}_{i_{j}}\cup B_{i_{j+1}}\cup \cdots\cup B_{i_{d}},\varphi]$, where the hat indicates omission. Choose two such vertices, and consider their common lower vertex (the one missing the two pieces missing in either of the chosen vertices). Applying the doctrine of no empty squares to these 3 vertices and the apex yields that the apex can be written as $[Z\cup B_{i_{1}}\cup\cdots\cup B_{i_{d}},\varphi]$, and thus the cube is filled.\\
	\qedhere
	\end{proof}

    \begin{remark} The argument above that ascending edges are unique per direction emphatically does not hold for descending edges. In fact, for any vertex of height at least 1, there are infinitely many descending edges per direction in which it can have a descending edge (given $[Z,\varphi]$, choose $\psi\in \mathcal{PH}_n$ such that $\varphi^{-1}\psi |_{Z}$ is a homeomorphism, with some nontrivial behavior in the piece to be removed, and such that $\varphi^{-1}\psi$ is rigid outside $Z$). It is helpful to keep in mind the following: when ascending, one can always choose coherent representatives from a minimal vertex; when descending, nothing is guaranteed.
    \end{remark}

    \begin{theorem}
        Let $X_n$ denote the Stein--Farley complex for $\H_{n}$. Then $X_n$ is CAT(0).
    \end{theorem}
 
	\begin{proof} By Proposition \ref{cubeComp}, $X_n$ is cube-complete. By construction, $X_n^{(1)}$ has no loops or bigons. As $X_n$ is contractible, it is simply connected, so item $(a)$ is covered. For item $(b)$, we need to show that 3-wheels can always be completed. There are four possibilities: the common vertex can be lowest height, highest height, or either of the intermediate heights. In the first case, we can have all three edges out of the common vertex use the same representative, each adding a piece/shift in a different end, so that the cube is completed by adding in all three at once. In the intermediate height cases, something similar occurs: we always have the lowest height vertex of the desired cube from which we can build upwards. It is useful to keep in mind the following arrangement of the 1-skeleton of a 3-cube; the solid lines indicate the 3-wheel, and the dashed lines indicate the rest of the 3-cube.
    
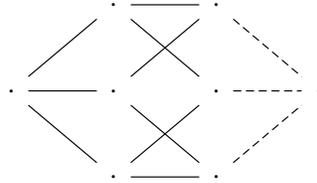
\begin{center} 
        \begin{tikzcd}
            & \cdot \arrow[r, no head] \arrow[dr, no head] & \cdot \arrow[dr, no head, dashed]& \\
            \cdot \arrow[ur, no head] \arrow[r, no head] \arrow[dr, no head] & \cdot \arrow[ur, no head] \arrow[dr, no head]& \cdot \arrow[r, no head, dashed] & \cdot\\
            & \cdot \arrow[ur, no head] \arrow[r, no head]& \cdot \arrow[ur, no head, dashed]& 
        \end{tikzcd}
    \end{center}
    
    \indent The difficulty is in the remaining case, where the common vertex is maximal in the cube. We begin with the diagram in Fig. \ref{3-wheel}. Our goal is to obtain a representative $[L_{g+2},\varphi\circ\rho_{k}]$ for the height $g+2$ vertex currently represented with both $\varphi'$ and $\varphi''$. This will imply the existence of a vertex $[L_{g}, \varphi\circ\rho_{i}\rho_{j}\rho_{k}]$, which realizes this 3-wheel in a 3-cube. To begin, we have the following lemma that justifies part of the diagram.
    \begin{lemma}
        Whenever $[L_g,\varphi\circ \rho_{i}] = [L_g,\varphi'\circ \rho_{j}]$ and $[L_{g+1},\varphi] = [L_{g+1},\varphi']$, we have that $i=j$.
    \end{lemma}

    \begin{proof}
    
    First, observe that the equality $[L_{g},\varphi\circ\rho_{i}] = [L_{g},\varphi'\circ\rho_{j}]$ means that $h=\rho_{i}^{-1}\varphi^{-1}\varphi'\rho_{j}$ takes $L_{g}$ to itself, and is rigid elsewhere; also, the equality $[L_{g+1},\varphi] = [L_{g+1},\varphi']$ means that $\varphi^{-1}\varphi'(L_{g+1})=L_{g+1}$. Assume $j\neq n$, and consider the piece $B^{1}_{j}$; let $\alpha$ be the boundary curve it shares with $\O$. What does $h$ do to $\alpha$? First,  $\rho_{j}$ deforms it into the center and stretches it into the $n$th end, making it an essential curve in $L_{g}$. Then $\varphi^{-1}\varphi'$ moves it around to some essential curve in $L_{g+1}$. Finally, $\rho_{i}^{-1}$ pushes and stretches it towards the $i$th end. As $B_{j}^{1}$ is not in $L_{g}$, it must be sent to a piece, and $h(\alpha)$ must be the intersection of this piece with $L_{g}$. As all representatives here do not permute ends, it must be that $h(\alpha)=\alpha$. If $i\neq j$, then $\rho_{i}$ fixes $\alpha$, so $\rho_{i}(\alpha)\neq \varphi^{-1}\varphi'\rho_{j}(\alpha)$, and $h(\alpha)\neq \alpha$. Hence we must have that $i=j$. In the case where $j=n$, equality is obvious. 
    \qedhere
    \end{proof}

    \indent Now, we seek the following equality: $[L_{g+2},\varphi\circ\rho_{k}] = [L_{g+2},\varphi'\circ\rho_{k}]$. That is, we wish to show that $\rho_{k}^{-1}\varphi^{-1}\varphi'\rho_{k}$ takes $L_{g+2}$ to itself, and is rigid elsewhere. We already know that this holds for $\rho_{i}^{-1}\varphi^{-1}\varphi'\rho_{i}$. The rigidity here occurs precisely when $\varphi^{-1}\varphi'$ does nothing untoward to whatever single piece gets pushed into the $n$th end, so it doesn't matter which index we choose! (If we have chosen the index $n$, there is no pushing. To resolve this, simply choose a different height $g+2$ vertex to work with. They can't share ends, as they are in squares with lower vertices, so we can apply Lemma \ref{UAE}.)\\
    \indent One small detail remains: we have a bottom vertex connecting to one of the height $g+1$ vertices, but does it connect to the other height $g+1$ vertices? Suppose not: then the cube between $[L_{g+3}, \varphi]$ and $[L_{g},\varphi\circ \rho_{i}\rho_{j}\rho_{k}]$ contains a square differing from one of the original 3 squares only at its height $g+1$ vertex (that is, at its lowest vertex). This however means that we would have two squares agreeing on 3 vertices, disagreeing on the 4th, which would contain a collapsed square. As this cannot occur (Lemma \ref{noCollSq}), our height $g$ vertex connects to each of the height $g+1$ vertices we started with.

    \begin{figure}[h!]
        \include{figure-2-wheel}
\vspace*{-1cm}        \caption{A 3-wheel with common vertex maximal.}
        \label{3-wheel}
    \end{figure}
	\qedhere
	\end{proof}

    \section{$\Sigma$-invariants of $P\H_{n}$}
    \indent The methods of this section are taken from \cite{Zar15} and \cite{Zar19}. They need only minimal modification to work in this setting, and are presented partly for the sake of having the entire argument in one place.
     We begin by noting that the abelianization of $P\H_{n}$ is $\Z^{n-1}$, and the abelianization map is $(\chi_{1},\dots,\chi_{n-1})$ (see Section 6 of \cite{ABKL}). As $\chi_{1}+\cdots+\chi_{n} = 0$, any (non-trivial) character $\chi$ can be written (up to renumbering the ends of $\Sigma_{n}$) in ascending standard form, i.e. $\chi = a_{1}\chi_{1}+\cdots+a_{n}\chi_{n}$, with $a_{1}\leq\cdots\leq a_{m(\chi)}<a_{m(\chi)+1}=\cdots = a_{n} =0$. We shall henceforth assume all characters are written in such form. Observe that $m(\chi)$ is defined to be the maximal index so that $a_{m(\chi)}<a_{n}=0$. Then our goal in this section is to prove the following.\\
    \begin{theorem}
        Let $\chi$ be a non-zero character of $P\H_{n}$. Then $[\chi]\in \Sigma^{m(\chi)-1}(P\H_{n})\setminus \Sigma^{m(\chi)}(P\H_{n})$.
    \end{theorem}
    
    \subsection{Inclusion}

    \indent We begin by showing the inclusion into the $(m(\chi)-1)$ layer of $\Sigma(P\H_{n})$.
    \begin{theorem}
        For any non-zero character $\chi\in \Hom(P\H_{n},\R)$, we have $[\chi]\in\Sigma^{m(\chi)-1}(P\H_{n})$.
    \end{theorem}

    \indent We recall here the form of the Morse lemma we shall use in this section.
    \begin{cor}
        Let $h=(\chi,f):Y\rightarrow \R\times\R$ be a Morse function. If $Y$ is $(m-1)$-connected and for every vertex $v\in Y^{\chi<q}$ the ascending link $\lk_{Y}^{h\uparrow}(v)$ is $(m-1)$-connected, then $Y^{q\leq \chi}$ is $(m-1)$-connected.
    \end{cor}
    
    \indent From now on, we write $X$ for $X_n$. The role of $Y$ shall be played by sublevel subcomplexes $X^{f\leq q}$ for sufficiently large $q$. Let $\chi = a_{1}\chi_{1} + \cdots + a_{n}\chi_{n}$ be a non-trivial character of $P\H_{n}$, written in ascending standard form. The function $h=(\chi,f):Y\rightarrow \R^{2}$ is a Morse function. Observe that between any two adjacent vertices, each basis character differs by $1$, $0$, or $-1$. We examine $h$-ascending links in $Y$.\\

    \begin{lemma}
    \label{inc link}
        Let $v=[L_{g},\varphi]$ be a vertex in $Y$. An adjacent vertex $w=[L_{g\pm 1},\psi]$ is in the $h$-ascending link of $v$ if and only if one of the two following conditions holds:
        \begin{itemize}
            \setlength\itemsep{-.4em}
            \item[$\bullet$] $w = [L_{g-1},\psi]$, with $v=[L_{g},\psi\rho_{i}^{-1}]$ where $i\leq m(\chi)$
            \item[$\bullet$] $w = [L_{g+1},\varphi\rho_{i}^{-1}]$, where $i\geq m(\chi) +1$.
        \end{itemize}
    \end{lemma} 
    \begin{proof}
        If $w$ is $f$-ascending from $v$, then $w=[L_{g+1},\varphi\rho_i^{-1}]$ for some $i$. If $i\geq m(\chi)+1$, then $\chi(w) = \chi(v)$, so $w$ is $h$-ascending from $v$. If $i\leq m(\chi)$, then $w$ is $\chi$-descending from $v$, hence $h$-descending from $v$. If $w$ is $f$-descending, then $v = [L_g,\psi\rho_i^{-1}]$ for some $i$. If $i\geq m(\chi)+1$, then we again have $\chi(w)=\chi(v)$, so that $w$ is $h$-descending from $v$. If $i\leq m(\chi)$, then $w$ is $\chi$-ascending from $v$, hence $h$-ascending.
    \end{proof}
    \begin{remark}
        It is worth observing a difference here from the situation in \cite{Zar15}. For the regular Houghton groups, there is one way to go ``up" and finitely many ways to go ``down" per end (with respect to $f$); here, while there is only one way to go up per end, there are infinitely many ways to go down. However, the $h$-ascending link is still a join of its intersection with the $f$-ascending and $f$-descending links, as these elements cannot share ends.
    \end{remark}

    \indent We shall need the following fact (see \cite[Section 5.2]{ABKL}): Let $v\in X_n$ be a vertex. If $f(v)\geq 2n$, then the $f$-descending link of $v$ is $(n-2)$-connected. Note that the version of this statement for the regular Houghton groups uses $f(v)\geq 2n-1$ (see \cite[Theorem 3.52]{Lee}); this discrepancy will result in most of the numbers in the following proposition being either one above or one below the corresponding numbers in \cite{Zar15}.\\
    \indent Now set $q=3n-2$, so $Y = X^{f\leq 3n-2}$. We have the following:
    \begin{prop}
        Let $v$ be a vertex in $Y$. Then $\lk_{Y}^{h\uparrow}(v)$ is $(m(\chi)-2)$-connected.
    \end{prop}

    \begin{proof} We have that $f(v)$ is between $0$ and $3n-2$. Suppose that $f(v)\leq 2n+m(\chi)-2$. Writing $v=[L_{g},\varphi]$, we have that there are $n-m(\chi)$ indices $i$ for which $v'=[L_{g+1},\varphi\circ\rho_{i}^{-1}]$ is $h$-ascending. As $(2n+m(\chi)-2)+(n-m(\chi)) = 3n-2$, the entire $f$-ascending link of $v$ in $X$ is contained in $Y$. Thus, the $f$-ascending part of the $h$-ascending link is an $(n-m(\chi)-1)$-simplex, which is contractible, so that $\lk_{Y}^{h\uparrow}(v)$ is contractible.\\
    \indent Suppose instead that $2n+m(\chi)-1\leq f(v)\leq 3n-2$, and thus that $Y$ does not contain the entire $f$-ascending part of the $h$-link of $v$. We still have its $(3n-f(v)-3)$-skeleton, which, being a skeleton of an $(n-m(\chi)-1)$-simplex, is $(3n-f(v)-4)$-connected. As $f(v)\geq 2n$, we have that the entire $f$-descending part of the $h$-ascending link in is $Y$. This is isomorphic to the $f$-descending link of a vertex of the same $f$-height as $v$ in $X_{m(\chi)}$, which we know is $(m(\chi)-2)$-connected so long as $f(v)\geq 2m(\chi)$. As $f(v)\geq 2n+m(\chi)-1$, we need that $2n-1\geq m(\chi)$, which is certainly true. The join is now $((3n-f(v)-3) + (m(\chi)-1))$-connected. As $f(v)\leq 3n-2$, we have that $(3n-f(v)+m(\chi)-4)\geq m(\chi)-2$, and thus that $\lk_{Y}^{h\uparrow}(v)$ is $(m(\chi)-2)$-connected.
    \qedhere

    \end{proof}

    \subsection{Exclusion}

    To demonstrate that the result of the last section is sharp, we shall need to employ different techniques. The main result of this section is
    \begin{theorem}
        For any non-zero character $\chi\in\Hom(P\H_{n},\R)$, we have $[\chi]\notin \Sigma^{m(\chi)}(\H_{n})$.
    \end{theorem}

    \indent We gather here the propositions and definitions from \cite{Zar19} that we will need, starting with the Strong Nerve Lemma:
    \begin{prop}
        Let $X$ be a CW-complex covered by subcomplexes $(X_{i})_{i\in I}$ and let $L$ be the nerve of the cover. Let $n\geq 1$. Suppose that any non-empty intersection $X_{i_{1}}\cap\cdots\cap X_{i_{r}}$ is $(n-r)$-connected. Then $H_{k}(X)\cong H_{k}(L)$ for all $k\leq n-1$, and $H_{n}(X)$ surjects onto $H_{n}(L)$.
    \end{prop}

    \begin{definition}
        For a set of indices $K\subseteq [n]$, consider the subcomplex $\bigcap_{i\in K} X^{\chi_{i}\leq 0}$ of $X$. Call any connected component of such a subcomplex a \textbf{$K$-blanket}. By a \textbf{blanket} we mean a $K$-blanket for some unspecified $K$.
    \end{definition}

    Recall that a subcomplex $Z$ of a CAT(0) cube complex $Y$ is \textbf{locally combinatorially convex} if every link in $Z$ of a vertex $z\in Z^{(0)}$ is a full subcomplex of the link of $z$ in $Y$, and \textbf{combinatorially convex} if it is connected and locally combinatorially convex. It is known that combinatorially convex implies CAT(0), hence contractible. In particular, this applies to connected components of locally combinatorially convex subcomplexes.\\

    \begin{lemma}
        For any $K$, $\bigcap_{i\in K}X^{\chi_{i}\leq 0}$ is locally combinatorially convex. Thus, blankets are combinatorially convex, and hence CAT(0) and contractible.
    \end{lemma}

    \begin{proof} It suffices to show that each $X^{\chi_{i}\leq 0}$ is locally combinatorially convex. Given a pair of adjacent vertices, we can write them as $v=[L_{g},\varphi]$ and $w=[L_{g+1},\varphi\circ\rho_{j}^{-1}]$ for some $j$. Then $\chi_{i}(w)-\chi_{i}(v) = \delta_{i,j}$. Thus, if $C$ is a cube containing $v$, and $w_{1},\dots,w_{k}$ are the vertices of $C$ adjacent to $v$, then the maximum and minimum values of $\chi_{i}$ on $C$ lie in $\{\chi_{i}(v), \chi_{i}(w_{1}),\dots,\chi_{i}(w_{k})\}$. Thus, whenever $v\in X^{\chi_{i}\leq 0}$ and all these $w_{j}$'s lie in the link of $v$ in $X^{\chi_{i}\leq 0}$, then the cube $C$ lies in $X^{\chi_{i}\leq 0}$. This implies that the link of $v$ in $X^{\chi_{i}\leq 0}$ is a full subcomplex of the link of $v$ in $X$.
    \qedhere
    \end{proof}

    \indent As in \cite{Zar19}, we have the immediate corollary that intersections of blankets are blankets for the union of the sets of indices. What follows is essentially identical to Zaremsky's approach for the Houghton groups, reproduced here with minor changes for the sake of completeness. Recall now the more general statement of the Morse lemma, Lemma \ref{Morse_Lemma}. For notational convenience, we write $X_{f\leq k}$ for $X^{f\leq k}$, and $X_{f\leq k}^{t\leq \chi}$ for the intersection $X_{f\leq k} \cap X^{t\leq \chi}$.
    \begin{lemma}
        If $X_{f\leq 3n-2}^{0\leq \chi}$ is not $(m(\chi)-1)$-connected, then $[\chi]\in \Sigma^{m(\chi)}(P\H_{n})^{c}$.
    \end{lemma}

    \begin{proof} If $[\chi]\in \Sigma^{m(\chi)}(P\H_{n})$, then the filtration $(X_{f\leq 3n-2}^{t\leq \chi})_{t\in \R}$ is essentially $(m(\chi)-1)$-connected. Every $h$-ascending link of a vertex in $X_{f\leq 3n-2}$ is $(m(\chi)-2)$-connected, so for any $s\leq t$ the inclusion $X_{f\leq 3n-2}^{t\leq \chi}\hookrightarrow X_{f\leq 3n-2}^{s\leq \chi}$ induces an isomorphism in $\pi_{k}$ for $k\leq m(\chi)-2$, and a surjection in $\pi_{m(\chi)-1}$. By assumption, for any $t$ there is some $s\leq t$ such that this inclusion induces a trivial map in $\pi_{k}$ for $k\leq m(\chi)-1$, implying that $X_{f\leq 3n-2}^{s\leq \chi}$ is $(m(\chi)-1)$-connected. Rescaling if necessary, we can assume that $s\in \chi(P\H_{n})$, and thus we can translate to obtain $X_{f\leq 3n-2}^{s\leq \chi}\cong X_{f\leq 3n-2}^{0\leq \chi}$, so that $X_{f\leq 3n-2}^{0\leq \chi}$ is itself $(m(\chi)-1)$-connected.
    \qedhere
    \end{proof}

    \indent In order to show that $X_{f\leq 3n-2}^{0\leq \chi}$ is not $(m(\chi)-1)$-connected, we will apply the Strong Nerve Lemma to a covering we now define. For $1\leq i\leq n$, let $\{Z_{i}^{\alpha}\}$ be the collection of $\{i\}$-blankets in $X$. The $\alpha$'s are indices in some set, and this index set is not itself important. Set also $$Y_{i}^{\alpha} = Z_{i}^{\alpha}\cap X_{f\leq 3n-2}^{0\leq \chi}.$$

    \begin{lemma}
        The $Y_{i}^{\alpha}$ with $1\leq i\leq m(\chi)$ cover $X_{f\leq 3n-2}^{0\leq \chi}$.
    \end{lemma}

    \begin{proof} As the $Z_{i}^{\alpha}$ are the connected components of the $X^{\chi_{i}\leq 0}$, it suffices to show that $$X^{0\leq\chi}\subseteq \bigcup_{i=1}^{m(\chi)} X^{\chi_{i}\leq 0}.$$ As $\chi = a_{1}\chi_{1}+\cdots + a_{m(\chi)}\chi_{m(\chi)}$, with all coefficients negative, any vertex $v\in X$ with $\chi(v)\geq 0$ must satisfy $\chi_{i}(v)\leq 0$ for some $i$. This implies the inclusion on vertices. Given a cube in $X^{0\leq \chi}$, let $v$ be its maximal vertex with respect to $f$. Then all the vertices $w$ of the cube satisfy $\chi_{i}(w)\leq \chi_{i}(v)$, and hence the entire cube lies in whichever $X^{\chi_{i}\leq 0}$ contains $v$.
    \qedhere
    \end{proof}

    \begin{lemma}
        Any non-empty intersection of subcomplexes of the form $Y_{i}^{\alpha}$ with $1\leq i\leq m(\chi)$ is $(m(\chi)-2)$-connected.
    \end{lemma}

    \begin{proof} To be non-empty, such an intersection can include at most one term $Y_{i}^{\alpha}$ for each $i$, and can thus be written $Y = Y_{i_{1}}^{\alpha_{1}}\cap \cdots \cap Y_{i_{r}}^{\alpha_{r}}$, with the $i_{j}$ all pairwise distinct. Let $Z = Z_{i_{1}}^{\alpha_{1}}\cap \cdots \cap Z_{i_{r}}^{\alpha_{r}}$, so that $Y= Z\cap X_{f\leq 3n-2}^{0\leq \chi}$. We apply Morse theoretic techniques to $Z$, this time using Lemma \ref{Morse_Lemma}. As $Z$ is an intersection of blankets, it is a blanket, and thus contractible. Given adjacent vertices $w=[L_{g},\varphi]$ and $v=[L_{g+1},\varphi\circ \rho_{i}^{-1}]$ with $v\in Z$, we have that $w\in Z$. Thus, for any vertex of $Z$, the entire $f$-descending link is in $Z$. As this is $(n-2)$-connected for $f(v) \geq2n$, we see that $Z_{f\leq 3n-2}$ is $(n-2)$-connected, and so is certainly $(m(\chi)-2)$-connected. Now consider $Y$ as $Z_{f\leq 3n-2}^{0\leq \chi}$. As before, the $h$-ascending link of a vertex $v$ is a join between its $f$-ascending and $f$-descending parts. The latter is in $Z$ for the same reasons as above; the former is in $Z$ because it consists of directions $i$ where $m(\chi)+1\leq i\leq n$, on which the considered $\chi_{i_{j}}$ are constant. Thus, the $h$-ascending link of $v$ is in $Z_{f\leq 3n-2}$. As $Z_{f\leq 3n-2}$ is $(m(\chi)-2)$-connected, the Morse lemma tells us that $Y$ is $(m(\chi)-2)$-connected.
    \qedhere
    \end{proof}

    Let $L$ be the nerve of the covering of $X_{f\leq 3n-2}^{0\leq \chi}$ by the $Y_{i}^{\alpha}$. Since $[\chi]\in \Sigma^{m(\chi)-1}(P\H_{n})$, we know that $X_{f\leq 3n-2}^{0\leq \chi}$ is $(m(\chi)-2)$-connected, so by the Strong Nerve Lemma $L$ is $(m(\chi)-2)$-connected. The final missing piece is to prove that $L$ is not $(m(\chi)-1)$-acyclic.

    \begin{lemma}
        The nerve $L$ is not $(m(\chi)-1)$-acyclic.
    \end{lemma}

    \begin{proof} Consider vertices corresponding to $Y_{i}^{\alpha}$ and $Y_{j}^{\beta}$. These vertices can only be adjacent if $i\neq j$, so $L$ is $(m(\chi)-1)$ dimensional. Thus, it suffices to exhibit a non-trivial $(m(\chi)-1)$-cycle. This will come from a collection of $2m(\chi)$ vertices, 2 for each $i\in\{1,\dots,m(\chi)\}$, labeled as $Y_{i}^{\epsilon_{i}}$, with $\epsilon\in\{1,2\}$, with the property $Y_{1}^{\epsilon_{i}}\cap\cdots\cap Y_{m(\chi)}^{\epsilon_{m(\chi)}} \neq \emptyset$. This will yield an embedded $(m(\chi)-1)$-sphere in $L$, which is homologically non-trivial for dimensional reasons.\\
    \indent Recall that $\O$ is the centerpiece of our surface. For each $i$, take $Y_{i}^{1}$ to be the $Y_{i}^{\alpha}$ containing $v_{0}=[\O, \id]$, and take $Y_{i}^{2}$ to be the $Y_{i}^{\alpha}$ containing $v_{i}=[\O, \varphi_{i}]$, where $\varphi_{i}$ is some non-trivial mapping class in the $i$th end; for the sake of definiteness, choose some essential simple closed curve in the standard piece, and let $\varphi_{i}$ be the Dehn twist about its image in $B_{i}^{1}$. Any intersection $Y_{1}^{\epsilon_{1}}\cap\cdots\cap Y_{m(\chi)}^{\epsilon_{m(\chi)}}$ includes $w =[\O,\prod \varphi_{i}]$, where the product (in an abitrary order) is taken over those $i$ with $\epsilon_{i}=2$, and is therefore nonempty.\\
    \indent It remains to show that $Y_{i}^{1}\neq Y_{i}^{2}$ for each $i$. It suffices to show that $Z_{i}^{1}\neq Z_{i}^{2}$. If these are equal (call them $Z_i$), then we can connect $v_{0}$ to $v_{i}$ via a path in $Z_i$. In $X$, one could assume that such a path is one on which $f$ first strictly increases, then strictly decreases; since $Z_i$ is combinatorially convex, this property holds for $Z_i$ as well. Since the path lies in $Z_i$, $\chi_{i}$ is non-positive on the whole path. By the uniqueness of ascending edges, this is a path of the form \begin{center}
        \begin{tikzcd}
            & {[Z,\id]} \arrow[r, Rightarrow, no head] & {[Z',\varphi_{i}]} &\\
             {[\O,\id]} \arrow[ur, no head, dotted] & & & {[\O, \varphi_{i}]} \arrow[ul, no head, dotted]\\ 
        \end{tikzcd}
    \end{center} where each ascending dotted line indicates a sequence of edges which only add pieces. Since $\chi_{i}(v_{0})= 0 = \chi_{i}(v_{i})$, none of the edges of the path can be obtained by adding a piece in the $i$th end. One of $Z$ and $Z'$ must have at least one piece in the $i$th end, as the transition map is $\varphi_{i}$, which has non-rigid behavior in the $i$th end. Thus, we have a contradiction. 
    \qedhere
    \end{proof}

    \section{Subgroups of maximal finiteness length }

    We use the notation $fl(G)=n$ to mean that $G$ is a group of type $F_{n}$ but not of type $F_{n+1}$. We know that if $G<H$ is finite index, then $fl(G)=fl(H)$. But what of the converse, that is, when does $fl(G)=fl(H)$ imply that $G$ is finite index in $H$? For Houghton groups, and for the (pure) surface Houghton groups, the answer is positive for sufficiently large subgroups, in particular for coabelian subgroups. We denote by $G'$ the commutator subgroup of a group $G$. Recall that the commutator subgroups of the Houghton group $H_{n}$ and the pure surface Houghton group $P\H_{n}$ are the finitely supported and compactly supported elements, respectively. (Really, we could say compactly supported in both cases, using the discrete topology for the $\N$-rays on which the Houghton group acts.) For Theorem \ref{fl the} and Proposition \ref{findex}, let $H$ be either $H_n$ or $P\H_n$. For $H_n$, this is a mild extension of \cite[Corollary 2.6]{Zar19}.
    \begin{theorem}
    \label{fl the}
        Let $G< H$ be a subgroup intersecting the commutator subgroup $H'$ in a finite index subgroup. If $fl(G) =n-1$, then $G$ is finite index in $H$.
    \end{theorem}
    \indent Note that it suffices to show this for groups containing the commutator subgroup. We begin by showing that when $G$ contains the commutator subgroup, the image of $G$ in the abelianization is a maximal rank sublattice, which will then imply finite index.
    \begin{prop}\label{findex}
        Let $G<H$ be as above, i.e. $fl(G)=n-1$, and $G$ contains the commutator $H'$. Write $F$ for the abelianization map to $\Z^{n-1}$. Then $F(G)$ is finite index in $\Z^{n-1}$.
    \end{prop}

    \begin{proof}
        We start by observing that $F = (\chi_{1},\dots,\chi_{n-1})$. As $\Sigma^{n-1}(H)$ is empty, the only way for $G$ to be of type $F_{n-1}$ is that it cannot be killed by any non-zero character of $H$. We shall construct a sequence of maps $F_{i} = (F_{i}^{1},\dots,F_{i}^{n-1}):H\rightarrow \Z^{n-1}$ with $F=F_{1}$, and a sequence of elements $g_{i}\in G$ such that $F_{n}^{j}(g_{i})$ is positive when $i=j$, and zero when $i<j$. Then $F_{n}(G)$ will be a maximal rank sublattice of $\Z^{n-1}$, obtainable from $F_{1}(G)$ by integer matrix transformations.\\
        \indent As $\chi_{1}(G)\neq 0$, we have some smallest positive integer $c_{1}\in\chi_{1}(G)$. Let $g_{1}\in\chi_{1}^{-1}(c_{1})\cap G$, and write $F_{1}(g_{1}) = (a_{1}^{1},\dots,a_{n-1}^{1})$ (note that $a_{1}^{1}=c_{1}$). Set $F_{2} = (\chi_{1}, a_{1}^{1}\chi_{2}-a_{2}^{1}\chi_{1},\dots, a_{1}^{1}\chi_{n-1}-a_{n-1}^{1}\chi_{1})$. Then $F_{2}(g_{1}) = (c_{1},0,\dots,0)$. By the same argument, we can choose $c_{2}>0$ minimal in the image of the character $a_{1}^{1}\chi_{2}-a_{2}^{1}\chi_{1}$, and then $g_{2}\in G$ in its preimage. Carrying on, we obtain $F_{i}$ from $F_{i-1}$ by modifying only the components from $i$ onwards, and build a collection $g_{1},\dots,g_{n-1}\in G$ such that the matrix obtained by applying $F_{n}$ to this collection is lower triangular with integer entries, with positive values on the main diagonal.
    \qedhere
    \end{proof}
        We now prove Theorem \ref{fl the}, using the characterization of the commutator subgroup as the elements of compact support. For $H_n$, the ends correspond to the rays in the obvious way. 
    \begin{proof}
        Suppose first that $G$ contains the commutator subgroup. Consider the map of coset spaces $\Psi:H/G\rightarrow \Z^{n-1}/(\textbf{c}\Z^{n-1})$ given by $hG\mapsto F(h)\textbf{c}\Z^{n-1}$. To see that this is well-defined, suppose $h_{1}G=h_{2}G$: then $h_{1}=h_{2}g$ for some $g\in G$, and $F(h_{1}) = F(h_{2})+F(g)$. As $F(g)\in \textbf{c}\Z^{n-1}$, we see that $\Psi(h_{1}G) = \Psi(h_{2}G)$. We now wish to show that $\Psi$ is injective. Suppose that $\Psi(h_{1}G)=\Psi(h_{2}G)$. Then $F(h_{1}h_{2}^{-1}) = F(h_{1})-F(h_{2})=F(g)$ for some $g\in G$. So our question becomes: does $F(h_{1}h_{2}^{-1})\in F(G)$ imply $h_{1}h_{2}^{-1} \in G$? The element $h_{1}h_{2}^{-1}g^{-1}$ will have no net translation in any end, and hence is compactly supported, and therefore is in $G$.\\
        \indent 
        \qedhere
    \end{proof}
    \indent As there are clearly subgroups of finite index in both cases (take the pre-image of a finite index subgroup of $\Z^{n-1}$), there is one loose end remaining: the infinite index case. Specifically, do there exist subgroups $G< H_n$ (or $G< P\H_n$) whose intersection with the commutator subgroup have infinite index, and such that $fl(G)=n-1$. First, we see that $G$ is necessarily infinite index in $H_n$ (or $P\H_n$), so all that remains is to check whether this case actually occurs. We thank Noel Brady for the proofs of the following lemmas.\\
    \begin{lemma}
        \label{index lemma}
        Suppose $G,K\leq H$ are subgroups. If $G$ is finite index in $H$, then $G\cap K$ is finite index in $K$.
    \end{lemma}

    \begin{proof}
        Suppose $[H:G] = m<\infty$. Write $H = Gh_{1}\cup\cdots\cup Gh_{m}$, as a disjoint union. For each index $i$ such that $K\cap Gh_{i}\neq \emptyset$, there is some $g_{i}\in G$ and $k_{i}\in K$ such that $k_{i}=h_{i}g_{i}$. Then we can write $$Gh_{i} = Gg_{i}^{-1}k_{i} = Gk_{i}.$$ Thus, $$K\cap Gh_{i} = K\cap Gk_{i} = Kk_{i}\cap Gk_{i} = (K\cap G)k_{i}.$$ Intersecting this with our original disjoint decomposition for $H$, We have $$K = (K\cap G)k_{1} \cup \cdots \cup (K\cap G)k_{s},$$ where $s\leq m$ is the number of indices for which $K\cap Gh_{i}\neq \emptyset$.
    \end{proof}
    \begin{lemma}
        Suppose that a group $H$ has an exhaustion by subgroups, i.e. there exists a nested chain of subgroups $K_1 \leq K_2 \leq \cdots \leq H$ such that $H$ is the union of all the $K_i$'s. Let $G\leq H$. Then $[H:G]$ is finite if and only if the sequence $([K_i :G\cap K_i])$ is eventually constant, in which case it is the limiting value.
    \end{lemma}

    \begin{proof}
        By the proof of Lemma \ref{index lemma}, we see that the sequence $([K_i :G\cap K_i])$ is bounded by $[H:G] = m$. It is not hard to see that it is non-decreasing, so we need only show that it can't stabilize below $m$. Choosing a decomposition of $H$ into $G$-cosets, say $H = Gh_{1}\cup \cdots \cup Gh_{m}$, we see that there must be some index $j$ such that $K_{j}$ contains all of $\{h_{1},\dots,h_{m}\}$. Therefore, the sequence of indices stabilizes at $m$.\\
        \indent For the converse, consider a decomposition $H = Gh_{1}\cup Gh_{2} \cup \cdots$. Intersecting this with $K_{i}$, we see as before that the index $[K_{i} : G\cap K_{i}]$ is the number of indices $j$ such that $Gh_{j} \cap K_i$ is non-empty. As every $h_{j}$ must be in some $K_{i}$, we see that if $[H:G]$ is infinite, then the sequence $([K_{i} : G\cap K_{i}])$ goes to infinity.
    \end{proof}
    \indent We return to our question on the existence of infinite index subgroups of Houghton groups with finiteness length $n-1$, We must now split off the ordinary Houghton group from the surface Houghton group. For $H_{n}$, the answer to our existence question is easily yes: there are even copies of $H_{n}$ itself as infinite index subgroups of $H_{n}$! Consider the stabilizer of a single point in $[n]\times \N$: ignoring the fixed point, and sliding its ray back by one to fill in, we have a natural bijection between $[n]\times \N$, on which $H_{n}$ acts, and $([n]\times \N)\setminus \{(i,j)\}$, on which $\Stab(i,j)$ acts, and we see that these actions are the same. (This fact was known to Houghton, see \cite{Houghton}.)\\
    \indent By the same argument, the subgroup fixing pointwise any finite set will be a copy of $H_{n}$, and the subgroup fixing (as a set, not pointwise) any finite set will be a finite extension of $H_{n}$.\\
    \indent In \cite{Yves}, Yves Cornulier defined a stronger failure of co-Hopfianness, which he called ``dis-cohopfian". Call a group $G$ \textbf{dis-co-Hopfian} if there is some injective homomorphism $\eta:G\rightarrow G$ such that the intersection of all iterated images of this homomorphism is trivial, i.e. $$\bigcap_{n=1}^{\infty} \eta^{n}(G) = \{e\}.$$ By taking the finite set being stabilized to be the first point in each $\N$-ray, we obtain the following:
    \begin{prop}
        \label{not cohopf}
        The Houghton groups $H_{n}$ are not co-Hopfian, and are in fact dis-co-Hopfian.
    \end{prop}
    \indent To extend the result on the existence of infinite index subgroups of finiteness length $n-1$ to surface Houghton groups, we shall work out more carefully the embedding $\Br H_{n}\hookrightarrow \H_{n}$ suggested by \cite{ABKL}. For the braided Houghton group, we take the asymptotically rigid mapping class group of the following surface, which we shall call $\dot{\Sigma}$ (see \cite{Funar} for details). Begin with a $2n$-gon as the center piece, and take a punctured square for the attached pieces (see Figure \ref{br def surf}). In \cite{GLU2}, it was shown that the braided Houghton group $\Br H_{n}$ is type $F_{n-1}$ but not type $FP_{n}$. This makes it a good candidate for an infinite index subgroup of $\H_{n}$ with $fl(\Br H_{n}) = fl(\H_{n})$.\\
    \begin{figure}
        \centering
        \include{braided-def}
        \vspace*{-.6cm}
        \caption{Left: the defining surface for $\Br H_n$; Right: the surface $\Sigma'$}
        \label{br def surf}
    \end{figure}
    \indent To obtain a homomorphism $\Br H_{n}\rightarrow \H_{n}$, we replace each puncture with a boundary component, and then pass to the double. Write $\Sigma'$ for the surface obtained by replacing punctures with boundary components. As for why this yields an injective homomorphism, consider the following diagram, where the upper row is exact: \begin{center}
        \begin{tikzcd}
            1\arrow[r] & K\arrow[r] & \Map(\Sigma')\arrow[r, "a"] \arrow[d, "b"] & \Map (\dot{\Sigma}) \arrow[r] & 1\\
            & & \Map(\Sigma)& &
        \end{tikzcd}
    \end{center} The failure of $b$ to be an injective homomorphism is precisely the subgroup generated by boundary parallel twists. The kernel of $a$ is generated by twists parallel to the compact boundary components. Therefore, we obtain an injective homomorphism $\Map(\dot{\Sigma}) \rightarrow \Map(\Sigma)$. Restricting to asymptotically rigid subgroups yields an injective homomorphism $\Br H_{n} \rightarrow \H_n$. As $\Br H_{n}$ acts trivially on the maximal (i.e. non-puncture) ends, this image lies in $P\H_{n}$.\\
    \indent The image of this homomorphism is a subgroup which fixes (up to isotopy) the multicurve defined by the images of the curve $\beta$ in Figure \ref{central curve} in each piece, via the canonical maps $\iota_{B}$ of Section 2.1. 

        \begin{figure}[h!]
        \centering
        \include{piece}
        \vspace*{-.5cm}
        \caption{The central curve $\beta$ in a piece which is fixed by the image of $\Br H_{n}$ in $P\H_{n}$}
        \label{central curve}
        \end{figure}

     All that remains is to see that this subgroup has infinite index in $P\H_{n}$. Consider its intersection with the mapping class group of any suited subsurface: here, it must fix the multicurve consisting of the above curves, and is therefore of infinite index. As infinite index in a subgroup implies infinite index in the full group, we have our result.\\
    \indent Between Theorem \ref{fl the}, Proposition \ref{not cohopf}, and the above discussion, we have proven the following:\\
    
    \begin{theorem}
         Let $H$ denote either the Houghton group, or the pure surface Houghton group, and suppose $G<H$ has $fl(G)=fl(H)$. Then $G$ is finite index in $H$ if and only if $G\cap H'$ is finite index in $H'$, where $H'$ denotes the commutator subgroup of $H$. Furthermore, there exist subgroups $G$ with $fl(G)=fl(H)$ of both finite and infinite index.
    \end{theorem}
    \indent We finish with further consideration of co-Hopfianness. It is easy to see that $\Br H_{n}$ is not co-Hopfian: there is an inclusion map from the defining surface of $\Br H_{n}$ to itself, sliding the first puncture in some end out, and with all elements moving things only on one side of the skipped puncture (see Figure \ref{br not cohopf}). As with the Houghton groups, doing such a move in all ends simultaneously yields a homomorphism whose iterated images act trivially on arbitrarily large compact subsurfaces. This yields:
    \begin{theorem}
        The braided Houghton group $\Br H_{n}$ is not co-Hopfian, and is in fact dis-co-Hopfian.
    \end{theorem}

    \begin{figure}
        \include{punctured-surface}
        \vspace*{-.6cm}
        \caption{Pushing one puncture ``off to the side"}
        \label{br not cohopf}
    \end{figure}

    \indent These approaches are not immediately available for the surface Houghton group, as there is no inclusion of surfaces which skips over a single genus. In fact, in light of the results of \cite{ALM}, if the pure surface Houghton group were to fail to be co-Hopfian, then it must fail either by a non-twist-preserving homomorphism, or by a homomorphism which is not the restriction of a homomorphism on the level of pure mapping class groups. This leaves us with the following question:

    \begin{question}
        Is the pure surface Houghton group $P\H_{n}$ co-Hopfian? If not, is it dis-co-Hopfian?
    \end{question}

\bibliographystyle{alpha}
\bibliography{References}

\bigskip
\bigskip
\bigskip

\begin{center}
    \begin{tabular}{|p{2.1in}@{\qquad\qquad\qquad}|p{2.1in}}
      Noah Torgerson 
      \newline
      Department of Mathematics
      \newline
      University of Oklahoma
      \newline
      \texttt{nmtorger@ou.edu}
      &
      Jeremy West 
      \newline
      Department of Mathematics
      \newline
      University of Oklahoma
      \newline
      \texttt{jeremy.west-1@ou.edu}
    \end{tabular}
  \end{center}

\end{document}

%% file: handle-shift.tex
\begin{center}
    \begin{tikzpicture}
        % trident edges

        \begin{scope}[shift={(0,0)}]
        \node (label-i) at (-2.02, .75) {$i$};
        \node (label-j) at (4.03, -.4) {$j$};
        \node (label-n) at (4.03, 1.8) {$n$};
        
        \draw (-2, .4) -- (0, .4) to[out=0, in =180] (1, 1.5) -- (4, 1.5) ;
        \draw (-2, -.4) -- (0, -.4) to[out=0, in=180] (1, -1.5) -- (4, -1.5);
        \draw (1, 0) to[out=90, in=180] (1.5, .75) -- (4, .75);
        \draw (1, 0) to[out=-90, in=180] (1.5, -.75) -- (4, -.75);

        \begin{scope}[shift={(2.2, 1.1)}, scale=.8]
            \draw (0, 0) arc (-30:-150:.5 and .25);
            \draw[shift={(-1.203,-1.08)}] ((1, 1) arc (30:150: .3 and .2);
        \end{scope}

        \begin{scope}[shift={(3.5, 1.1)}, scale=.8]
            \draw (0, 0) arc (-30:-150:.5 and .25);
            \draw[shift={(-1.203,-1.08)}] ((1, 1) arc (30:150: .3 and .2);
        \end{scope}

        \begin{scope}[shift={(2.2, -1.1)}, scale=.8]
            \draw (0, 0) arc (-30:-150:.5 and .25);
            \draw[shift={(-1.203,-1.08)}] ((1, 1) arc (30:150: .3 and .2);
        \end{scope}

        \begin{scope}[shift={(3.5, -1.1)}, scale=.8]
            \draw (0, 0) arc (-30:-150:.5 and .25);
            \draw[shift={(-1.203,-1.08)}] ((1, 1) arc (30:150: .3 and .2);
        \end{scope}

        \begin{scope}[shift={(-.3, 0)}, scale=.8]
            \draw (0, 0) arc (-30:-150:.5 and .25);
            \draw[shift={(-1.203,-1.08)}] ((1, 1) arc (30:150: .3 and .2);
        \end{scope}

        \draw[color = {rgb,255:red,254; green,97; blue,0}] (-.2, .4) arc (80:-80:.2 and .4);

        \draw[color = {rgb,255:red,30; green,136; blue,229}] (-1.4, .4) arc (80:-80:.2 and .4);

        \draw[color = {rgb,255:red,255; green,176; blue,0}] (1.3, 1.5) arc (80:-80:.2 and .4);

        \draw[color = {rgb,255:red,220; green,38; blue,127}] (1.3,-.7) arc (80:-80:.2 and .4);
        \end{scope}

        \node[scale=1.5] (map-arrow) at (5,0) {$\overset{\rho_i}{\longrightarrow}$};

\begin{scope}[shift={(8,0)}]
        \node (label-i) at (-2.02, .75) {$i$};
        \node (label-j) at (4.03, -.4) {$j$};
        \node (label-n) at (4.03, 1.8) {$n$};

        \draw (-2, .4) -- (0, .4) to[out=0, in =180] (1, 1.5) -- (4, 1.5) ;
        \draw (-2, -.4) -- (0, -.4) to[out=0, in=180] (1, -1.5) -- (4, -1.5);
        \draw (1, 0) to[out=90, in=180] (1.5, .75) -- (4, .75);
        \draw (1, 0) to[out=-90, in=180] (1.5, -.75) -- (4, -.75);
    
        \begin{scope}[shift={(2.2, 1.1)}, scale=.8]
            \draw (0, 0) arc (-30:-150:.5 and .25);
            \draw[shift={(-1.203,-1.08)}] ((1, 1) arc (30:150: .3 and .2);
        \end{scope}

        \begin{scope}[shift={(3.5, 1.1)}, scale=.8]
            \draw (0, 0) arc (-30:-150:.5 and .25);
            \draw[shift={(-1.203,-1.08)}] ((1, 1) arc (30:150: .3 and .2);
        \end{scope}

        \begin{scope}[shift={(2.2, -1.1)}, scale=.8]
            \draw (0, 0) arc (-30:-150:.5 and .25);
            \draw[shift={(-1.203,-1.08)}] ((1, 1) arc (30:150: .3 and .2);
        \end{scope}

        \begin{scope}[shift={(3.5, -1.1)}, scale=.8]
            \draw (0, 0) arc (-30:-150:.5 and .25);
            \draw[shift={(-1.203,-1.08)}] ((1, 1) arc (30:150: .3 and .2);
        \end{scope}

        \begin{scope}[shift={(-.3, 0)}, scale=.8]
            \draw (0, 0) arc (-30:-150:.5 and .25);
            \draw[shift={(-1.203,-1.08)}] ((1, 1) arc (30:150: .3 and .2);
        \end{scope}

        \draw[color = {rgb,255:red,255; green,97; blue,0}, dashed] 
            (.1, .4) to[out=0, in=180] 
            (1.6, 1.4) to [out=0, in=90] 
            (2.3, 1) to [out=-90, in=-180]
            (1.6, .85) to[out=180, in=0]
            (.1, -.4);

        \draw[color = {rgb,255:red,255; green,97; blue,0}, radius=.1] 
            (.1, -.4) to[out=20, in=180] 
            (1.6, .9) to[out=0, in = -90]
            (2.5, 1) to[out=90, in = 0]
            (1.8, 1.25) to[out=180, in=-30] (.1, .4);

        \draw[color = {rgb,255:red,30; green,136; blue,229}] (-.2, .4) arc (80:-80:.2 and .4);

        \draw[color = {rgb,255:red,255; green,176; blue,0}] (2.5, 1.5) arc (80:-80:.2 and .375);
        
        \draw[color = {rgb,255:red,220; green,38; blue,127}] (1.3,-.7) arc (80:-80:.2 and .4);

        \end{scope}

    \end{tikzpicture}

\end{center}

%% file: 2.6-diagram.tex
\begin{center}
    \begin{tikzpicture}

         \pgfdeclarelayer{boxes}
        \pgfdeclarelayer{nodes}
        \pgfdeclarelayer{connections}
        \pgfsetlayers{connections,boxes,nodes,main}

        % labels
        \def\nodeALabel{$[Z, \varphi]$}
        \def\nodeBLabel{$[Z \cup B, \varphi]$}

        % nodes
        \node (A) at (0, 0) {\nodeALabel};
        \node (B) at (2.3, 0) {\nodeBLabel};

        % node connections
        \begin{pgfonlayer}{connections}
        \draw[thick] (A) -- (2.6,0);
        \end{pgfonlayer}

        % gray boxes
        \begin{pgfonlayer}{boxes}

        \filldraw[fill=white, rounded corners=.5pt] ($(0, 0) - (.6, .5)$) rectangle ($(0, 0) + (.6,.5)$);    
        \filldraw[draw opacity=.3, opacity=.1, fill=black, rounded corners=.5pt] ($(0, 0) - (.6, .5)$) rectangle ($(0, 0) + (.6,.5)$);    

        \filldraw[fill=white, rounded corners=.5pt] ($(2.3, 0) - (.9, .5)$) rectangle ($(2.3, 0) + (.9,.5)$);    
        \filldraw[draw opacity=.3, opacity=.1, fill=black, rounded corners=.5pt] ($(2.3, 0) - (.9, .5)$) rectangle ($(2.3, 0) + (.9,.5)$);    
        \end{pgfonlayer}

        %labels
        \def\nodeCLabel{$[\sigma(Z), \varphi \circ \sigma^{-1}]$}
        \def\nodeDLabel{$[\sigma(Z \cup B), \varphi \circ \sigma^{-1}]$}
        \def\nodeELabel{$[\sigma(Z) \cup \sigma(B), \varphi \circ \sigma^{-1}]$}

        %nodes
        \begin{pgfonlayer}{nodes}
        \node (C) at (5, 0) {\nodeCLabel};
        \node (D) at (9, 0) {\nodeDLabel};
        \node (E) at (9, -1.5) {\nodeELabel};
        \end{pgfonlayer}

        % node connections
        \begin{pgfonlayer}{connections}
        \draw[thick] (4,0) -- (D);
        \end{pgfonlayer}

        % gray boxes
        \begin{pgfonlayer}{boxes}
        \filldraw[fill=white, rounded corners=.5pt] ($(5, 0) - (1.3, .5)$) rectangle ($(5, 0) + (1.3,.5)$);   
        \filldraw[draw opacity=.3, opacity=.1, fill=black, rounded corners=.5pt] ($(5, 0) - (1.3, .5)$) rectangle ($(5, 0) + (1.3,.5)$);   

        \filldraw[fill=white, rounded corners=.5pt] ($(9, -.75) - (2.2, 1.2)$) rectangle ($(9, -.75) + (2.2,1.2)$);    
        \filldraw[draw opacity=.3, opacity=.1, fill=black, rounded corners=.5pt] ($(9, -.75) - (2.2, 1.2)$) rectangle ($(9, -.75) + (2.2,1.2)$);    
        \end{pgfonlayer}

        % node equalities
            \draw[transform canvas={xshift=-1.5pt}, ultra thick] (D) -- (E);
            \draw[transform canvas={xshift=+1.5pt}, ultra thick] (D) -- (E);

    \end{tikzpicture}

\end{center}

%% file: 2.7-diagram.tex
\begin{center}
    \begin{tikzpicture}

         \pgfdeclarelayer{boxes}
        \pgfdeclarelayer{nodes}
        \pgfdeclarelayer{connections}
        \pgfsetlayers{connections,boxes,nodes,main}

        % labels
        \def\nodeALabel{$[L_g, \varphi]$}
        \def\nodeBLabel{$[L_g \cup B_i^1, \varphi]$}
        \def\nodeCLabel{$[L_{g+1},\varphi \circ \rho_i^{-1}]$}

        % nodes
        \begin{pgfonlayer}{nodes}
        \node (A) at (0, -1.5) {\nodeALabel};
        \node (B) at (2.7, -.2) {\nodeBLabel};
        \node (C) at (2.7, -1.7) {\nodeCLabel};
        \end{pgfonlayer}

        % node connections
        \begin{pgfonlayer}{connections}
        \draw[thick] (-.1,-1.5) -- (2,-.2);
        \draw[thick] (-.1,-1.5) -- (2.9,-1.5);
        \end{pgfonlayer}

        % gray boxes
        \begin{pgfonlayer}{boxes}
        \filldraw[fill=white, rounded corners=.5pt] ($(0, -1.5) - (.6, .5)$) rectangle ($(0, -1.5) + (.6,.5)$);    
        \filldraw[draw opacity=.3, opacity=.1, fill=black, rounded corners=.5pt] ($(0, -1.5) - (.6, .5)$) rectangle ($(0, -1.5) + (.6,.5)$);    
        \filldraw[fill=white, rounded corners=.5pt] ($(2.7, -.95) - (1.3, 1.1)$) rectangle ($(2.7, -.95) + (1.3, 1.1)$);    
        \filldraw[draw opacity=.3, opacity=.1, fill=black, rounded corners=.5pt] ($(2.7, -.95) - (1.3, 1.1)$) rectangle ($(2.7, -.95) + (1.3, 1.1)$);    
        \end{pgfonlayer}

        % node equalities
            \draw[transform canvas={xshift=-1.5pt}, ultra thick] (B) -- (C);
            \draw[transform canvas={xshift=+1.5pt}, ultra thick] (B) -- (C);

    \end{tikzpicture}
\end{center}

%% file: 3.5-digram.tex
\begin{center}
    \begin{tikzpicture} 

        \pgfdeclarelayer{boxes}
        \pgfdeclarelayer{nodes}
        \pgfdeclarelayer{connections}
        \pgfsetlayers{connections,boxes,nodes,main}

        % labels
        \def\nodeALabel{$[Z, \varphi]$}
        \def\nodeBLabel{$[Z', \varphi']$}
        \def\nodeCLabel{$[Z \cup B, \varphi]$}
        \def\nodeDLabel{$[Z' \cup B', \varphi']$}

        % nodes
        \begin{pgfonlayer}{nodes}
        \node (A) at (0, 1.5) {\nodeALabel};
        \node (B) at (0, 0) {\nodeBLabel};
        \node (C) at (2.5, 1.5) {\nodeCLabel};
        \node (D) at (2.5, 0) {\nodeDLabel};
        \end{pgfonlayer}

        % node connections
        \begin{pgfonlayer}{connections}
        \draw[thick] (A) -- (C);
        \draw[thick] (-0.1, 0) -- (2.5, 0);
        \end{pgfonlayer}

        % gray boxes
        \begin{pgfonlayer}{boxes}
        \filldraw[fill=white, rounded corners=.5pt] ($(0, .75) - (.6, 1.1)$) rectangle ($(0, .75) + (.6, 1.1)$);    
        \filldraw[draw opacity=.3, opacity=.1, fill=black, rounded corners=.5pt] ($(0, .75) - (.6, 1.1)$) rectangle ($(0, .75) + (.6, 1.1)$);    
        \filldraw[draw opacity=.1, fill=white, rounded corners=.5pt] ($(2.5, .75) - (1, 1.1)$) rectangle ($(2.5, .75) + (1, 1.1)$);    
        \filldraw[opacity=.1, fill=black, draw opacity=.1, rounded corners=.5pt] ($(2.5, .75) - (1, 1.1)$) rectangle ($(2.5, .75) + (1, 1.1)$);    
        \end{pgfonlayer}

        % node equalities
            \draw[transform canvas={xshift=-1.5pt}, ultra thick] (A) -- (B);
            \draw[transform canvas={xshift=+1.5pt}, ultra thick] (A) -- (B);

            \draw[transform canvas={xshift=-1.5pt}, ultra thick, opacity=.3, dashed] (C) -- (D);
            \draw[transform canvas={xshift=+1.5pt}, ultra thick, opacity=.3, dashed] (C) -- (D);

            \node (E) at (2.8, .75) {?};

    \end{tikzpicture}
\end{center}

%% file: figure-1-collapsed-square.tex
\begin{center}
    \begin{tikzpicture}

        \pgfdeclarelayer{boxes}
        \pgfdeclarelayer{nodes}
        \pgfdeclarelayer{connections}
        \pgfsetlayers{connections,boxes,nodes,main}

        % labels
        \def\nodeALabel{$[L_g, \varphi]$}

        \def\nodeBLabel{$[L_{g+1}, \varphi\rho_{i_1}^{-1}]$}
        \def\nodeCLabel{$[L_{g+1}, \psi\rho_{j_1}^{-1}]$}
        \def\nodeDLabel{$[L_{g+1}, \varphi\rho_{i_2}^{-1}]$}
        \def\nodeELabel{$[L_{g+1}, \psi\rho_{j_2}^{-1}]$}

        \def\nodeFLabel{$[L_{g}, \psi]$}

        % nodes
        \begin{pgfonlayer}{nodes}
        \node (A) at (-3, 1.75) {\nodeALabel};
        
        \node (B) at (0, 3) {\nodeBLabel};
        \node (C) at (0, 1.5) {\nodeCLabel};
        \node (D) at (0, -.5) {\nodeDLabel};
        \node (E) at (0, -2) {\nodeELabel};
        
        \node (F) at (3, -.25) {\nodeFLabel}; 
        \end{pgfonlayer}

        % node connections
        \begin{pgfonlayer}{connections}
        \draw[thick] (A) -- (B);
        \draw[thick] (A) -- (D);

        \draw[thick] (F) -- (C);
        \draw[thick] (F) -- (E);
        \end{pgfonlayer}

        % gray boxes
        
        \begin{pgfonlayer}{boxes}

        \filldraw[fill=white, rounded corners=.5pt] ($(0, 2.25) - (1.1, 1.4)$) rectangle ($(0, 2.25) + (1.1, 1.4)$);    
        \filldraw[draw opacity=.3, opacity=.1, fill=black, rounded corners=.5pt] ($(0, 2.25) - (1.1, 1.4)$) rectangle ($(0, 2.25) + (1.1, 1.4)$);    
        \filldraw[fill=white, rounded corners=.5pt] ($(0, -1.25) - (1.1, 1.4)$) rectangle ($(0, -1.25) + (1.1, 1.4)$); 
        \filldraw[draw opacity=.3, opacity=.1, fill=black, rounded corners=.5pt] ($(0, -1.25) - (1.1, 1.4)$) rectangle ($(0, -1.25) + (1.1, 1.4)$); 

        \filldraw[fill=white, rounded corners=.5pt] ($(-3, 1.75) - (.7, .5)$) rectangle ($(-3, 1.75) + (.7, .5)$); 
        \filldraw[draw opacity=.3, opacity=.1, fill=black, rounded corners=.5pt] ($(-3, 1.75) - (.7, .5)$) rectangle ($(-3, 1.75) + (.7, .5)$); 
        \filldraw[fill=white, rounded corners=.5pt] ($(3, -.25) - (.7, .5)$) rectangle ($(3, -.25) + (.7, .5)$); 
        \filldraw[draw opacity=.3, opacity=.1, fill=black, rounded corners=.5pt] ($(3, -.25) - (.7, .5)$) rectangle ($(3, -.25) + (.7, .5)$); 
        \end{pgfonlayer}

        % node equalities
            \draw[transform canvas={xshift=-1.5pt}, ultra thick] (B) -- (C);
            \draw[transform canvas={xshift=+1.5pt}, ultra thick] (B) -- (C);

            \draw[transform canvas={xshift=-1.5pt}, ultra thick] (D) -- (E);
            \draw[transform canvas={xshift=+1.5pt}, ultra thick] (D) -- (E);

    \end{tikzpicture}
\end{center}

%% file: figure-2-wheel.tex
\begin{center}
    \begin{tikzpicture}
        \pgfdeclarelayer{boxes}
        \pgfdeclarelayer{nodes}
        \pgfdeclarelayer{connections}
        \pgfsetlayers{connections,boxes,nodes,main}

        \def\nodeAA{$[L_{g+1}, \varphi \circ \rho_i\rho_j]$}
        \def\nodeAB{$[L_{g+1}, \varphi' \circ \rho_i\rho_k]$}
        \def\nodeAC{$[L_{g+1}, \varphi'' \circ \rho_j\rho_k]$}

        \def\nodeBA{$[L_{g+2}, \varphi \circ \rho_i]$}
        \def\nodeBB{$[L_{g+2}, \varphi' \circ \rho_i]$}
        \def\nodeBC{$[L_{g+2}, \varphi \circ \rho_j]$} 
        \def\nodeBD{$[L_{g+2}, \varphi'' \circ \rho_j]$}
        \def\nodeBE{$[L_{g+2}, \varphi' \circ \rho_k]$}
        \def\nodeBF{$[L_{g+2}, \varphi'' \circ \rho_k]$}

        \def\nodeCA{$[L_{g+3}, \varphi]$}
        \def\nodeCB{$[L_{g+3}, \varphi']$}
        \def\nodeCC{$[L_{g+3}, \varphi'']$}

        % column 1 nodes
        \begin{pgfonlayer}{boxes}
        \begin{scope}[yscale=1.5]
            \def\xcoord{-1}

             \filldraw[fill=white, rounded corners=.5pt] ($(\xcoord, 1) - (1.4, .5)$) rectangle ($(\xcoord, 1) + (1.4,.5)$);    
             \filldraw[fill=black, opacity=.1,, rounded corners=.5pt] ($(\xcoord, 1) - (1.4, .5)$) rectangle ($(\xcoord, 1) + (1.4,.5)$);    
      %  \node (AA) at (\xcoord, 1) {\nodeAA};

       % \node (AB) at (\xcoord, -.5) {\nodeAB};
       
            \filldraw[fill=white, rounded corners=.5pt] ($(\xcoord, -.5) - (1.4, .5)$) rectangle ($(\xcoord, -.5) + (1.4,.5)$);    
            \filldraw[draw opacity=.3, opacity=.1, fill=black, rounded corners=.5pt] ($(\xcoord, -.5) - (1.4, .5)$) rectangle ($(\xcoord, -.5) + (1.4,.5)$);    
            %\node (AC) at (\xcoord, -2) {\nodeAC};

            \filldraw[fill=white, rounded corners=.5pt] ($(\xcoord, -2) - (1.4, .5)$) rectangle ($(\xcoord, -2) + (1.4,.5)$);    
            \filldraw[draw opacity=.3, opacity=.1, fill=black, rounded corners=.5pt] ($(\xcoord, -2) - (1.4, .5)$) rectangle ($(\xcoord, -2) + (1.4,.5)$);    
        \end{scope}

        % column 2 nodes
        \begin{scope}
            \def\xcoord{3}
            \def\pairoffset{0}

       %     \node (BA) at ($(\xcoord, 2.5) + (0, \pairoffset)$) {\nodeBA};
       %     \node (BB) at ($(\xcoord, 1.5) - (0, \pairoffset)$) {\nodeBB};

            \filldraw[fill=white, rounded corners=.5pt] ($(\xcoord, 1.5) - (1.3, .5)$) rectangle ($(\xcoord, 2.5) + (1.3,.5)$);    
            \filldraw[draw opacity=.3, opacity=.1, fill=black, rounded corners=.5pt] ($(\xcoord, 1.5) - (1.3, .5)$) rectangle ($(\xcoord, 2.5) + (1.3,.5)$);    
        %    \node (BC) at ($(\xcoord, 0) + (0, \pairoffset)$) {\nodeBC};
         %   \node (BD) at ($(\xcoord, -1) - (0, \pairoffset)$) {\nodeBD};

        \filldraw[fill=white, rounded corners=.5pt] ($(\xcoord, -1) - (1.3, .5)$) rectangle ($(\xcoord, 0) + (1.3,.5)$);    
        \filldraw[draw opacity=.3, opacity=.1, fill=black, rounded corners=.5pt] ($(\xcoord, -1) - (1.3, .5)$) rectangle ($(\xcoord, 0) + (1.3,.5)$);    
        %    \node (BE) at ($(\xcoord, -2.5) + (0, \pairoffset)$) {\nodeBE};
        %    \node (BF) at ($(\xcoord, -3.5) - (0, \pairoffset)$) {\nodeBF};

        \filldraw[fill=white, rounded corners=.5pt] ($(\xcoord, -3.5) - (1.3, .5)$) rectangle ($(\xcoord, -2.5) + (1.3,.5)$);    
        \filldraw[draw opacity=.3, opacity=.1, fill=black, rounded corners=.5pt] ($(\xcoord, -3.5) - (1.3, .5)$) rectangle ($(\xcoord, -2.5) + (1.3,.5)$);    
        \end{scope}

        % column 3 nodes
        \begin{scope}
            \def\xcoord{7}

            \node (CA) at (\xcoord, .5) {\nodeCA};
            \node (CB) at (\xcoord, -.5) {\nodeCB};
            \node (CC) at (\xcoord, -1.5) {\nodeCC};

            \filldraw[fill=white, rounded corners=.5pt] ($(\xcoord, -1.5) - (1.2, .5)$) rectangle ($(\xcoord, .5) + (1.2,.5)$);    
            \filldraw[draw opacity=.3, opacity=.1, fill=black, rounded corners=.5pt] ($(\xcoord, -1.5) - (1.2, .5)$) rectangle ($(\xcoord, .5) + (1.2,.5)$);    
        \end{scope}

        \end{pgfonlayer}

        \begin{pgfonlayer}{nodes}
            \begin{scope}[yscale=1.5]
            \def\xcoord{-1}

            \node (AA) at (\xcoord, 1) {\nodeAA};
            \node (AB) at (\xcoord, -.5) {\nodeAB};
            \node (AC) at (\xcoord, -2) {\nodeAC};
        
            \end{scope}

            \begin{scope}
                \def\xcoord{3}
                \def\pairoffset{0}

            \node (BA) at ($(\xcoord, 2.5) + (0, \pairoffset)$) {\nodeBA};
            \node (BB) at ($(\xcoord, 1.5) - (0, \pairoffset)$) {\nodeBB};

            \node (BC) at ($(\xcoord, 0) + (0, \pairoffset)$) {\nodeBC};
            \node (BD) at ($(\xcoord, -1) - (0, \pairoffset)$) {\nodeBD};

            \node (BE) at ($(\xcoord, -2.5) + (0, \pairoffset)$) {\nodeBE};
            \node (BF) at ($(\xcoord, -3.5) - (0, \pairoffset)$) {\nodeBF};

        \end{scope}

            \begin{scope}
                \def\xcoord{7}

            \node (CA) at (\xcoord, .5) {\nodeCA};
            \node (CB) at (\xcoord, -.5) {\nodeCB};
            \node (CC) at (\xcoord, -1.5) {\nodeCC};

        \end{scope}

        \end{pgfonlayer}

        \begin{pgfonlayer}{connections}

        % column 1 to column 2 connections
      %  \draw (AA) -- (BA);
        
        \draw[teal, thick, transform canvas={yshift=+1.5pt}](AA) -- (BA);
       % \draw (AA) -- (BC);

        \draw[teal, thick, transform canvas={yshift=+1.5pt}](AA) -- (BC);

        %\draw (AB) -- (BB);
        
        \draw[purple, thick, transform canvas={yshift=-1.5pt}](AB) -- (BB);
      %  \draw (AB) -- (BE);

        \draw[purple, thick, transform canvas={yshift=+1.5pt}](AB) -- (BE);

    %    \draw (AC) -- (BD);

        \draw[orange, thick, transform canvas={yshift=-1.5pt}]($(AC) + (-.2, 0)$) -- (BD);

      %  \draw (AC) -- (BF);

        \draw[orange, thick, transform canvas={yshift=-1.5pt}](-1,-3) -- (3,-3);

        % column 2 to column 3 connections
      %  \draw ($(BA)!.5!(BB)$) -- (CA);

        \draw[teal, thick, transform canvas={yshift=+1.5pt}] ($(BA)!.5!(BB)$) -- (CA); 
       \draw[purple, thick, transform canvas={yshift=-1.5pt}]($(BA)!.5!(BB)$) -- (CA);

    %    \draw ($(BC)!.5!(BD)$) -- (CB);

        \draw[teal, thick, transform canvas={yshift=+1.5pt}]($(BC)!.5!(BD)$) -- (CB);

        \draw[orange, thick, transform canvas={yshift=-1.5pt}]($(BC)!.5!(BD)$) -- (CB);

     %   \draw ($(BE)!.5!(BF)$) -- (CC);

       \draw[purple, thick, transform canvas={yshift=+1.5pt}]($(BE)!.5!(BF)$) -- (CC);

        \draw[orange, thick, transform canvas={yshift=-1.5pt}]($(BE)!.5!(BF)$) -- (CC);

        \end{pgfonlayer}

       % column 1 equality 
      %\draw[transform canvas={xshift=-1.5pt}, ultra thick] (AA) -- (AB) -- (AC);
       % \draw[transform canvas={xshift=1.5pt}, ultra thick] (AA) -- (AB) -- (AC);

       % column 2 equality
        \draw[transform canvas={xshift=-1.5pt}, ultra thick] (BA) --(BB);
        \draw[transform canvas={xshift=+1.5pt}, ultra thick] (BA) --(BB);

        \draw[transform canvas={xshift=-1.5pt}, ultra thick] (BC) --(BD);
        \draw[transform canvas={xshift=+1.5pt}, ultra thick] (BC) --(BD);

        \draw[transform canvas={xshift=-1.5pt}, ultra thick] (BE) --(BF);
        \draw[transform canvas={xshift=+1.5pt}, ultra thick] (BE) --(BF);

        % column 3 equality
        \draw[transform canvas={xshift=-1.5pt}, ultra thick] (CA) -- (CB) -- (CC);
        \draw[transform canvas={xshift=+1.5pt}, ultra thick] (CA) -- (CB) -- (CC);

    \end{tikzpicture}

\end{center}

%% file: braided-def.tex
\begin{center}
    \def \globalscale {0.900000}
    \begin{tikzpicture}[y=1cm, x=1cm, yscale=\globalscale,xscale=\globalscale, every node/.append style={scale=\globalscale}, inner sep=0pt, outer sep=0pt]
        \begin{scope}
        \draw (90:1) -- (150:1);
        \draw (210:1) -- (270:1);
        \draw (330:1)-- (30:1);
        
        \draw (90:1) --++ (60:2);
        \draw (30:1) --++ (60:2);

        \draw (150:1) --++ (180:2);
        \draw (210:1) --++ (180:2);

        \draw (270:1) --++ (-60:2);
        \draw (330:1) --++ (-60:2);

        \draw[gray, dashed] (30:1) -- (90:1);
        \draw[gray, dashed] (30:1) ++ (60:1) --++ (150:1);
        
        \draw[gray, dashed] (150:1) -- (210:1);
        \draw[gray, dashed] (150:1) ++ (180:1) --++(270:1);

        \draw[gray, dashed] (270:1) -- (330:1);
        \draw[gray, dashed] (270:1) ++ (-60:1) --++ (30:1);

        \draw[fill=black] (60:1.4) circle (1pt);
        \draw[fill=black] (60:2.4) circle (1pt);

        \draw[fill=black] (180:1.4) circle (1pt);
        \draw[fill=black] (180:2.4) circle (1pt);

        \draw[fill=black] (-60:1.4) circle (1pt);
        \draw[fill=black] (-60:2.4) circle (1pt);
        \end{scope}

         \begin{scope}[shift={(6,0)}]
        \draw (90:1) -- (150:1);
        \draw (210:1) -- (270:1);
        \draw (330:1)-- (30:1);
        
        \draw (90:1) --++ (60:2);
        \draw (30:1) --++ (60:2);

        \draw (150:1) --++ (180:2);
        \draw (210:1) --++ (180:2);

        \draw (270:1) --++ (-60:2);
        \draw (330:1) --++ (-60:2);

        \draw[gray, dashed] (30:1) -- (90:1);
        \draw[gray, dashed] (30:1) ++ (60:1) --++ (150:1);
        
        \draw[gray, dashed] (150:1) -- (210:1);
        \draw[gray, dashed] (150:1) ++ (180:1) --++(270:1);

        \draw[gray, dashed] (270:1) -- (330:1);
        \draw[gray, dashed] (270:1) ++ (-60:1) --++ (30:1);

             \draw[fill=white] (60:1.4) circle (3pt);
             \draw[fill=white] (60:2.4) circle (3pt);

             \draw[fill=white] (180:1.4) circle (3pt);
             \draw[fill=white] (180:2.4) circle (3pt);

             \draw[fill=white] (-60:1.4) circle (3pt);
             \draw[fill=white] (-60:2.4) circle (3pt);
        \end{scope}

    \end{tikzpicture}
\end{center}

%% file: piece.tex
\definecolor{cfb0000}{RGB}{251,0,0}

\def \globalscale {2.800000}
\begin{tikzpicture}[y=1cm, x=1cm, yscale=\globalscale,xscale=\globalscale, every node/.append style={scale=\globalscale}, inner sep=0pt, outer sep=0pt]

    \node[scale = .4] (central curve) at (2.25, 3.15) {$\beta$};

  \path[draw=black,line join=round,line width=0.0074cm,miter limit=4.0] (3.5397, 3.1454).. controls (3.5662, 3.1454) and (3.5877, 3.0484) .. (3.5877, 2.9288) -- (3.5877, 2.9288).. controls (3.5877, 2.8092) and (3.5653, 2.7135) .. (3.5388, 2.7135);

  \path[draw=black,line join=round,line width=0.0074cm,miter limit=4.0,dash pattern=on 0.0074cm off 0.0074cm] (3.5429, 3.1439).. controls (3.5164, 3.1439) and (3.4949, 3.0469) .. (3.4949, 2.9273) -- (3.4949, 2.9273).. controls (3.4949, 2.8076) and (3.5173, 2.712) .. (3.5438, 2.712);

  \path[draw=black,line join=round,line width=0.01cm,miter limit=4.0] (1.4203, 2.712)arc(90.0006:0.0:0.0517 and -0.2168)arc(360.0:269.9994:0.0517 and -0.2168);

  \path[draw=black,line join=round,line width=0.01cm,miter limit=4.0] (1.4209, 3.1453).. controls (1.5531, 3.1445) and (1.7229, 3.1434) .. (1.7859, 3.1737).. controls (1.8489, 3.204) and (1.8974, 3.2981) .. (2.028, 3.3658).. controls (2.1587, 3.4334) and (2.3025, 3.4567) .. (2.4788, 3.4567).. controls (2.6552, 3.4567) and (2.7813, 3.4363) .. (2.9261, 3.3599).. controls (3.0709, 3.2835) and (3.1249, 3.1838) .. (3.179, 3.1578).. controls (3.2331, 3.1319) and (3.411, 3.139) .. (3.5391, 3.1442);

  \path[draw=black,line join=round,line width=0.01cm,miter limit=4.0] (1.4197, 2.7129).. controls (1.5523, 2.7135) and (1.7244, 2.7143) .. (1.7863, 2.6945).. controls (1.8483, 2.6747) and (1.9056, 2.6006) .. (2.0158, 2.5436).. controls (2.1261, 2.4865) and (2.3024, 2.4294) .. (2.4788, 2.4295).. controls (2.6552, 2.4295) and (2.8316, 2.4865) .. (2.9418, 2.5436).. controls (3.0521, 2.6007) and (3.1149, 2.6819) .. (3.1657, 2.6984).. controls (3.2165, 2.7148) and (3.4049, 2.7148) .. (3.5372, 2.7148);

  \path[draw=black,line join=round,line width=0.01cm,miter limit=4.0] (2.0954, 3.0156).. controls (2.0954, 2.9867) and (2.0954, 2.9579) .. (2.1303, 2.929).. controls (2.1651, 2.9001) and (2.2349, 2.8712) .. (2.322, 2.8568).. controls (2.4091, 2.8424) and (2.5137, 2.8424) .. (2.6009, 2.8568).. controls (2.688, 2.8713) and (2.7577, 2.9001) .. (2.7926, 2.929).. controls (2.8274, 2.9579) and (2.8274, 2.9867) .. (2.8274, 3.0156);

  \path[draw=black,line join=round,line width=0.009cm,miter limit=4.0] (2.1867, 2.8892).. controls (2.2313, 2.9328) and (2.2604, 2.9612) .. (2.334, 2.9792).. controls (2.4076, 2.9972) and (2.5181, 2.9972) .. (2.5917, 2.9792).. controls (2.6653, 2.9612) and (2.7021, 2.9252) .. (2.7389, 2.8892);

  \path[draw=cfb0000,line width=0.0501cm,miter limit=4.0,dash pattern=on 0.0201cm off 0.0201cm] (2.7929, 2.9338) -- (2.7929, 2.9338) -- (2.7929, 2.9338).. controls (2.8207, 2.8442) and (2.5796, 2.784) .. (2.4557, 2.7837).. controls (2.3302, 2.7833) and (2.0773, 2.8493) .. (2.1134, 2.9338).. controls (2.1144, 2.9362) and (2.1136, 2.939) .. (2.1137, 2.9416);

  \path[draw=cfb0000,line width=0.0501cm,miter limit=4.0] (2.1184, 2.9384).. controls (2.1309, 3.0019) and (2.2519, 3.0657) .. (2.4547, 3.0666).. controls (2.674, 3.0676) and (2.7769, 2.993) .. (2.7919, 2.932);

  \path[draw=black,line join=round,line width=0.01cm,miter limit=4.0] (1.4203, 2.712)arc(89.9994:180.0:0.0517 and -0.2168)arc(180.0:270.0006:0.0517 and -0.2168);

\end{tikzpicture}

%% file: punctured-surface.tex
\begin{center}
    \def \globalscale {0.900000}
    \begin{tikzpicture}[y=1cm, x=1cm, yscale=\globalscale,xscale=\globalscale, every node/.append style={scale=\globalscale}, inner sep=0pt, outer sep=0pt]
        \begin{scope}
        \draw (90:1) -- (150:1);
        \draw (210:1) -- (270:1);
        \draw (330:1)-- (30:1);
        
        \draw (90:1) --++ (60:2);
        \draw (30:1) --++ (60:2);

        \draw (150:1) --++ (180:2);
        \draw (210:1) --++ (180:2);

        \draw (270:1) --++ (-60:2);
        \draw (330:1) --++ (-60:2);

        \draw[gray, dashed] (30:1) -- (90:1);
        \draw[gray, dashed] (30:1) ++ (60:1) --++ (150:1);
        
        \draw[gray, dashed] (150:1) -- (210:1);
        \draw[gray, dashed] (150:1) ++ (180:1) --++(270:1);

        \draw[gray, dashed] (270:1) -- (330:1);
        \draw[gray, dashed] (270:1) ++ (-60:1) --++ (30:1);

        \draw[fill=black] (60:1.4) circle (1pt);
        \draw[fill=black] (60:2.4) circle (1pt);

        \draw[fill=black] (180:1.4) circle (1pt);
        \draw[fill=black] (180:2.4) circle (1pt);

        \draw[fill=black] (-60:1.4) circle (1pt);
        \draw[fill=black] (-60:2.4) circle (1pt);
        \end{scope}

         \begin{scope}[shift={(6,0)}]
        \draw (90:1) -- (150:1);
        \draw (210:1) -- (270:1);
        \draw (330:1)-- (30:1);
        
        \draw (90:1) --++ (150:1) --++ (60:1) --++ (-30:1) --++ (60:1);
             \draw[fill=black, opacity=.2] (90:1) --++ (150:1) --++ (60:1) --++ (-30:2) --++ (60:-1);
        \draw (30:1) --++ (60:2);

        \draw (150:1) --++ (180:2);
        \draw (210:1) --++ (180:2);

        \draw (270:1) --++ (-60:2);
        \draw (330:1) --++ (-60:2);

        \draw[gray, dashed] (30:1) -- (90:1);
        \draw[gray, dashed] (30:1) ++ (60:1) --++ (150:1);
        
        \draw[gray, dashed] (150:1) -- (210:1);
        \draw[gray, dashed] (150:1) ++ (180:1) --++(270:1);

        \draw[gray, dashed] (270:1) -- (330:1);
        \draw[gray, dashed] (270:1) ++ (-60:1) --++ (30:1);

             \draw[fill=black] (60:1.4) ++ (150:1) circle (1pt);
             \draw[fill=black] (60:2.4) circle (1pt);

             \draw[fill=black] (180:1.4) circle (1pt);
             \draw[fill=black] (180:2.4) circle (1pt);

             \draw[fill=black] (-60:1.4) circle (1pt);
             \draw[fill=black] (-60:2.4) circle (1pt);
        \end{scope}

    \end{tikzpicture}
\end{center}